\documentclass[11pt,a4paper]{article}
\usepackage[T1]{fontenc}
\usepackage{amssymb,amsmath,amsthm}
\usepackage[margin=2.5cm]{geometry}
\usepackage[english]{babel}
\usepackage[latin1]{inputenc}
\usepackage{enumerate}
\usepackage[all]{xy}
\usepackage{multicol}
\usepackage{varioref}
\usepackage{color}

\theoremstyle{plain}
\newtheorem{theo}{Theorem}[section]
\newtheorem{lemma}[theo]{Lemma}
\newtheorem{cor}[theo]{Corollary}
\newtheorem{prop}[theo]{Proposition}

\newtheorem{lemmaa}{Lemma}

\theoremstyle{definition}
\newtheorem{defi}[theo]{Definition}

\theoremstyle{remark}
\newtheorem{rem}[theo]{Remark}
\newtheorem{ex}[theo]{Example}
\newtheorem{nota}[theo]{Notation}

\newcommand{\tr}{{\rm tr}}
\newcommand{\Tr}{{\rm Tr}}
\newcommand{\End}{{\rm End}}
\newcommand{\Hom}{{\rm Hom}}
\newcommand{\Com}{{\rm Comod}}
\newcommand{\id}{{\rm id}}
\newcommand{\Aaut}{A_{\textbf{aut}}}
\newcommand{\A}{\mathcal{A}}
\newcommand{\gO}{\mathcal{O}}

\newcommand{\R}{\mathcal{R}}
\newcommand{\N}{\mathbb{N}}
\newcommand{\C}{\mathbb{C}}
\newcommand{\Z}{\mathbb{Z}}
\newcommand{\amb}[1]{Let us show that the ambiguity #1 is resolvable. On the first hand we have:}
\newcommand{\Amb}[1]{let us show that the ambiguity #1 is resolvable. On the first hand we have:}

\begin{document}
\title{Quantum automorphism groups and $SO(3)$-deformations}
\author{Colin MROZINSKI}
\date{}
\maketitle
\begin{center}
\textit{{Laboratoire de~Math\'ematiques (UMR 6620)}, Universit\'e Blaise Pascal, \\ Complexe universitaire des C\'ezeaux, 63171 Aubi\`ere Cedex, France}
\end{center}
\begin{center}
{\tt colin.mrozinski@math.univ-bpclermont.fr}
\end{center}
\begin{abstract}
We show that any compact quantum group having the same fusion rules as the ones of $SO(3)$ is the quantum automorphism group of a pair $(A, \varphi)$, where $A$ is a finite dimensional $C^*$-algebra endowed with a homogeneous faithful state. We also study the representation category of the quantum automorphism group of $(A, \varphi)$ when $\varphi$ is not necessarily positive, generalizing some known results,
and we discuss the possibility of classifying the cosemisimple (not necessarily compact) Hopf algebras whose corepresentation semi-ring is isomorphic to that of $SO(3)$.
\end{abstract}

\section{Introduction and main results}

The quantum automorphism group of a measured finite dimensional $C^*$-algebra $(A,\varphi)$ (i.e. a finite-dimensional $C^*$-algebra $A$ endowed with a faithful state $\varphi$)  has been defined by Wang in \cite{Wang1} as the universal object in the category of compact quantum groups acting on $(A,\varphi)$. The corresponding compact Hopf algebra is denoted by $\Aaut(A,\varphi)$.

The structure of $\Aaut(A,\varphi)$ depends on the choice of the measure $\varphi$, and the representation theory of this quantum group is now well understood \cite{Ban2, BanFussCat}, provided a good choice of $\varphi$ has been done, namely that $\varphi$ is a $\delta$-form (we shall say here that $\varphi$ is homogeneous, and that $(A, \varphi)$ is a homogeneous measured $C^*$-algebra). Banica's main result in \cite{Ban2, BanFussCat} is that if $\varphi$ is homogeneous and $\dim(A)\geq 4$, then $\Aaut(A,\varphi)$ has the same corepresentation semi-ring as $SO(3)$. See also \cite{BrannRedQAG}. The result can be further extended to show that the corepresentation category of  $\Aaut(A,\varphi)$ is monoidally equivalent to the representation category of a quantum $SO(3)$-group at a well chosen parameter, see \cite{DeRVaVenn}.
 
Then a natural question, going back to \cite{Ban2, BanFussCat} and formally asked in \cite{BanicaSurvey}, is whether any compact quantum group with the same fusion rules as $SO(3)$ is the quantum automorphism group of an appropriate measured finite-dimensional $C^*$-algebra. The main result in this paper is a positive answer to this question.

\begin{theo}\label{ClassTh}
Let $H$ be a compact Hopf algebra with corepresentation semi-ring isomorphic to that of $SO(3)$. Then there exists a finite dimensional homogeneous measured $C^*$-algebra $(A, \varphi)$  with $\dim(A)\geq 4$ such that $H \simeq \Aaut(A,\varphi)$.
\end{theo}

Recall that if $G$ is a reductive algebraic group, a $G$-deformation is a cosemisimple Hopf algebra $H$ such that $\R^+(H) \simeq \R^+(\gO(G))$, where $\R^+$ denotes the corepresentation semi-ring. The problem of the classification of $G$-deformations has been already studied for several algebraic groups: see \cite{Wo, Ban1, PM, Bi1} for $SL(2)$, \cite{Ohn, Mroz1} for  $GL(2)$, and \cite{Ohn2} for $SL(3)$.
Thus Theorem \ref{ClassTh} provides the full description of the compact $SO(3)$-deformations.  

The next natural step is then to study the non-compact $SO(3)$-deformations. For this purpose we study the comodule category of $\Aaut(A,\varphi)$ with $\varphi$ non necessarily positive and give a generalization of the results from \cite{Ban2, BanFussCat, Bichon3, DeRVaVenn} (together with independent proof of these results), as follows (see Section 2 for the relevant definitions).

\begin{theo}\label{SO3comod}
 Let $(A,\varphi)$ be a finite dimensional, semisimple algebra endowed with a normalizable measure $\varphi$, with $\dim A \geq 4$. Then there exists  a $\C$-linear equivalence of monoidal categories $$\Com(\Aaut(A,\varphi))\simeq^{\otimes}\Com(\gO(SO_{q}(3)))$$ between the comodule categories of $\Aaut(A,\varphi)$ and $\gO(SO_{q}(3))$ respectively, for some well-chosen $q \in \C^*$.
\end{theo} 

We have not been able to show that all $SO(3)$-deformations arise as quantum automorphism groups as in the previous theorem. However see Section 5 for partial results in this direction. Note that the monoidal reconstruction theorem of Tuba-Wenzl \cite{TubaWenzl}, which discuss the related but non equivalent problem of determining the braided semisimple tensor categories of type B, cannot be used in our setting, where the existence of a braiding is not assumed.

This paper is organized as follows: in Sec. 2, we fix some notations and definitions, state some basic facts about compact Hopf algebras, finite dimensional algebras and we recall the construction of the quantum automorphism group of a finite dimensional, semisimple, measured algebra. Theorem \ref{ClassTh} is proved in Sec. 3, thanks to a careful study of the fusion rules of $SO(3)$. In Sec. 4, we prove Theorem \ref{SO3comod} by building a cogroupoid linking these Hopf algebras and studying its connectedness and in Sec. 5, we prove some classification results about Hopf algebras having a corepresentation semi-ring isomorphic to that of $SO(3)$.

\section{Preliminaries}

\subsection{Compact Hopf algebras}

Let us recall the definition of a compact Hopf algebra (see \cite{KS}): 

\begin{defi}
 \begin{enumerate}
  \item A Hopf $*$-algebra is a Hopf algebra $H$ which is also a $*$-algebra and such that the comultiplication is a $*$-homomorphism.
  \item If $x=(x_{ij})_{1\leq i,j\leq n}\in M_n(H)$ is a matrix with coefficient in $H$, the matrix $(x^*_{ij})_{1\leq i,j\leq n}$ is denoted by $\overline{x}$, while $\overline{x}^t$, the transpose matrix of $\overline{x}$, is denoted by $x^*$. The matrix $x$ is said to be unitary if $x^*x=I_n=xx^*$.
  \item A Hopf $*$-algebra is said to be a compact Hopf algebra if for every finite-dimensional $H$-comodule with associated multiplicative matrix of coefficients $x\in M_n(H)$, there exists $K\in GL_n(\C)$ such that the matrix $KxK^{-1}$ is unitary.
 \end{enumerate}
\end{defi}
Compact Hopf algebras correspond to Hopf algebras of representative functions on compact quantum groups. In this paper, we only consider compact quantum groups at the level of compact Hopf algebras.

\subsection{Finite dimensional semisimple algebras}

In this subsection, we collect some facts about finite dimensional algebras and introduce some convenient notations and definitions.

\begin{defi}
 Let $A$ be an algebra. A \emph{measure} on $A$ is a linear form $\varphi : A \to \C$ such that the induced bilinear form $\varphi \circ m : A\otimes A \to \C$ is non degenerate. A \emph{measured algebra} is a pair $(A, \varphi)$ where $A$ is an algebra and $\varphi$ is a measure on $A$.
\end{defi}

The following definition will be useful:

\begin{defi}\label{DefNorm}
  Let $(A, \varphi)$ be a finite dimensional semisimple measured algebra. Let $\tilde{\delta} : \C \to A\otimes A$ be the dual map of the bilinear form $\varphi \circ m$. We define the application $$\tilde{\varphi} := (\varphi \circ m) \circ (m \otimes \id_A) \circ (\id_A \otimes \tilde{\delta}) : A \to \C. $$  Then 
\begin{itemize}
  \item We say that $(A, \varphi)$ is \emph{homogeneous} if there exists $\lambda_A \in \C^*$ such that $\tilde{\varphi} = \lambda_A \varphi$.
  \item We say that $(A, \varphi)$ is \emph{normalized} if $\tilde{\varphi} = \varphi(1_A) \varphi$.
  \item We say that $(A, \varphi)$ is \emph{normalizable} if $(A, \varphi)$ is homogeneous and $\varphi(1_A)\neq 0$.
 \end{itemize}
We say that a measured $C^*$-algebra $(A,\varphi)$ is homogeneous (resp. normalized) if $\varphi$ is homogeneous (resp. normalized) and positive.
\end{defi}

\begin{rem}
 It is a consequence of Cauchy-Schwarz inequality that a homogeneous and positive measure on a $C^*$-algebra is always normalizable.
\end{rem}

\begin{ex}
The canonical trace used by Banica in \cite{Ban2} is homogeneous, as well as the $\delta$-forms from \cite{BanFussCat}.
\end{ex}

Finite dimensional semisimple measured algebras can be described in term of a more concret object.

\begin{defi}
 Let $0< d_1 \leq \dots \leq d_{n}$ be some nonzero positive integers. We call a \emph{multimatrix} an element $E=(E_1,\dots, E_{n})\in \underset{\lambda=1}{\overset{n}{\bigoplus}} GL_{d_{\lambda}}(\C)$. If $\tr(-)$ is the usual trace, we denote:
 
$$\begin{aligned}
  &\Tr(E):=\underset{\lambda=1}{\overset{n}{\sum}} \tr(E_{\lambda}),  & &\tr_E:=\oplus \tr(E_{\lambda}^{-1t}-)\\
  & E^{-1}:=(E_1^{-1},\dots, E_{n}^{-1}),  & & E^{t}:=(E_1^{t},\dots, E_{n}^{t})
\end{aligned}$$
and we say that $E$ is \emph{positive} if each $E_{\lambda}$ is positive.
\end{defi}

Let us recall some well known results.

\begin{lemma}\label{Wedderburn}
 Let $A$ be a finite dimensional $C^*$-algebra. Then there exist some nonzero positive integers $0< d_1 \leq \dots \leq d_n$ such that $$A\simeq \underset{\lambda=1}{\overset{n}{\bigoplus}} M_{d_{\lambda}}(\C)$$

If $\varphi : A \to \C$ is a measure, then there exists a multimatrix $E=(E_1,\dots, E_{n})\in \underset{\lambda=1}{\overset{n}{\bigoplus}} GL_{d_{\lambda}}(\C)$ such that $\varphi=\tr_E$, and $\varphi$ is positive and faithful if and only if $E_{\lambda}$ is positive for all $ 1 \leq \lambda \leq n$.
\end{lemma}

\begin{nota}
 For a multimatrix $E\in \underset{\lambda=1}{\overset{n}{\bigoplus}} GL_{d_{\lambda}}(\C)$, we denote by $A_E$ the algebra $\underset{\lambda=1}{\overset{n}{\bigoplus}} M_{d_{\lambda}}(\C)$, and we denote by $(e_{kl,\lambda})_{kl,\lambda}$ its canonical basis.
\end{nota}

The following lemma translates the definition \ref{DefNorm} in term of multimatrices.

\begin{lemma}
 Let $E=(E_1,\dots, E_{n})\in \underset{\lambda=1}{\overset{n}{\bigoplus}} GL_{d_{\lambda}^E}(\C)$ be a multimatrix. Then:
\begin{enumerate}
 \item $(A_E, \tr_E)$ is homogeneous if and only if $\tr(E_{\lambda})=\tr (E_{\mu})\neq 0$, for all $ \lambda, \mu=1,\dots, n$
 \item $(A_E, \tr_E)$ is normalized if and only if $\Tr(E^{-1})=\tr(E_{\lambda})$, for all $ \lambda=1,\dots, n$
 \item $(A_E, \tr_E)$ is normalizable if and only if there exists $\xi \in \C^*$ such that $\Tr((\xi E)^{-1})=\tr(\xi E_{\lambda})\neq 0$, for all $ \lambda=1,\dots, n$
\end{enumerate}
\end{lemma}

\begin{proof}
The linear map $\tilde{\delta} : \C \to A\otimes A$ is given by $$\tilde{\delta}(1)=\sum_{\lambda=1}^{n} \sum_{k,l,r=1}^{d_{\lambda}^E} E_{lr,\lambda} e_{kl,\lambda}\otimes e_{rk,\lambda}.$$ Then $\tilde{\varphi}$ is is given by $$\tilde{\varphi}(e_{kl,\lambda})= \tr(E_\lambda) E^{-1}_{kl,\lambda}=\tr(E_\lambda)\varphi(e_{kl,\lambda})$$ Hence, $\tilde{\varphi}$ coincides with $\tr_E$ up to a nonzero scalar if and only if $E$ is homogeneous, which proves the first claim. The second claim follows from $$\tr_E(1_A)=\Tr(E^{-1}).$$ The last claim is now immediate.
\end{proof}

From now, we say that a multimatrix $E$, as well as the induced measure $\tr_E$, is homogeneous (resp. normalized, normalizable) if the measured algebra $(A_E, \tr_E)$ is homogeneous (resp. normalized, normalizable).

\subsection{The Quantum automorphism group $\Aaut(A,\varphi)$}

We can now recall the construction of the quantum automorphism group $\Aaut(A,\varphi)$ for a finite dimensional, semisimple, measured algebra $(A, \varphi)$ from \cite{Wang1}.

\begin{prop}
 Let $(A,\varphi)=(A_E, \tr_E)$ be a finite dimensional, semisimple, measured algebra and let $(e_{ij,\lambda})_{(ij,\lambda)}$ be its canonical basis. The quantum automorphism group $\Aaut(A,\varphi)$ is defined as follows. As an algebra, $\Aaut(A,\varphi)$ is the universal algebra with generators $X_{kl,\mu}^{ij,\lambda}$ $(1\leq \lambda, \mu \leq n$, $1\leq i,j \leq d_{\lambda}$, $1\leq k,l\leq d_{\mu})$ submitted to the relations

\begin{align*}
&\sum_{q=1}^{d_{\nu}} X_{ij,\lambda}^{rq,\nu}X_{kl,\mu}^{qs,\nu}=\delta_{\lambda \mu} \delta_{jk} X_{il,\mu}^{rs,\nu}, & &\sum_{\mu=1}^{n}\sum_{k=1}^{d_{\mu}} X_{kk,\mu}^{ij,\lambda}=\delta_{ij},\\
&\sum_{\mu=1}^{n}\sum_{k,l=1}^{d_{\mu}}E^{-1}_{kl,\mu} X_{ij,\lambda}^{kl,\mu}=E^{-1}_{ij,\lambda}, & &\sum_{r,s=1}^{d_{\lambda}} E_{rs,\lambda} X_{ir,\lambda}^{kp,\mu}X_{sj,\lambda}^{ql,\nu}=\delta_{\mu \nu}E_{pq,\mu}X_{ij,\lambda}^{kl,\mu}.
\end{align*}

It has a natural Hopf algebra structure given by 

$$\Delta (X_{ij,\lambda}^{kl,\mu})=\sum_{\nu=1}^{n} \sum_{p,q=1}^{d_{\nu}} X_{pq,\nu}^{kl,\mu} \otimes X_{ij,\lambda}^{pq,\nu}, \ \varepsilon (X_{ij,\lambda}^{kl,\mu})= \delta_{ik}\delta_{jl}\delta_{\lambda \mu}, \ S(X_{ij,\lambda}^{kl,\mu})=\sum_{r=1}^{d_{\lambda}}\sum_{s=1}^{d_{\mu}} E^{-1}_{rj,\lambda} E_{ls,\mu} X_{sk,\mu}^{ri,\lambda}$$
and the algebra map $\alpha_A: A \to A\otimes \Aaut (A,\varphi)$ defined by  $$\alpha_A(e_{ij,\lambda})=\sum_{\mu=1}^{n}\sum_{p,q=1}^{d_{\mu}} e_{pq,\mu}\otimes X^{pq,\mu}_{ij,\lambda}$$ is a coaction on $A$ such that $\varphi$ is equivariant.

If $(H,\alpha)$ is a Hopf algebra coacting on $(A, \varphi)$ with an algebra morphism $\alpha : (A,\varphi) \to (A,\varphi)\otimes H$, then there exists a Hopf algebra morphism $f:\Aaut(A,\varphi) \to H$ such that $(f\otimes \id_A)\circ \alpha_A=\alpha$.

If moreover $E$ is positive, $\Aaut(A,\varphi)$ is a compact Hopf algebra for the $*$-structure $$(X_{kl,\mu}^{ij,\lambda})^* =X_{lk,\mu}^{ji,\lambda}.$$ Then $\alpha_A$ is a $*$-morphism and if $H$ is a compact Hopf algebra, $f$ is also a $*$-morphism.
\end{prop}

\begin{ex}\label{ExSO3}
 \begin{enumerate}
\item If $X_n$ is the set consisting of $n$ distinct points and $\psi$ is the uniform probability measure on $X_n$, then $\Aaut(C(X_n),\psi)$ is the quantum permutation group on $n$ points, see \cite{Wang1}.
\item Let $(A,\varphi)=(M_2(\C), \tr)$. Then $\Aaut(A,\varphi)\simeq \mathcal{O}(PSL_2(\C)) \simeq \mathcal{O}(SO_3(\C))$. See \cite{Ban2, BichNat}.
\item\label{ExSOq3} Let $q\in \C^*$ and $$F_q:=\left(\begin{array}{cc}
q^{-1} & 0 \\
0 & q 
\end{array}\right)$$ Denote $\tr_q:=\tr_{F_q}$. Then we have $\Aaut(M_2(\C),\tr_q) \simeq \mathcal{O}(SO_{q^{1/2}}(3))$. 
\end{enumerate}
\end{ex}

\begin{rem}\label{RemNorm}
\begin{enumerate}
 \item  Let $(A_E,\tr_E)$ be a finite dimensional, semisimple, measured algebra, where $E=(E_1,\dots,E_{\lambda_0},\dots, E_{n})$. Then we have a Hopf algebra surjection $$\Aaut(A_E,\tr_E) \to \Aaut(A_{E_{\lambda_0}},\tr_{E_{\lambda_0}})$$ given by $$X^{ij,\lambda}_{kl,\mu} \mapsto \left\{ \begin{array}{cc}
\delta_{\lambda \mu} X^{ij}_{kl} & \text{when } \lambda=\lambda_0 \\
\delta_{\lambda \mu} \delta_{ik} \delta_{jl} & \text{otherwise} 
\end{array}\right. $$
 \item In view of the relations defining $\Aaut(A_E,\tr_E)$, we have $\Aaut(A_E,\tr_E)=\Aaut(A_{\xi E},\tr_{\xi E})$ for all $\xi \in \C^*$. Then if $E\in \underset{\lambda=1}{\overset{n}{\bigoplus}} GL_{d_{\lambda}^E}(\C)$ is normalizable, there exists $F\in \underset{\lambda=1}{\overset{n}{\bigoplus}} GL_{d_{\lambda}^E}(\C)$ normalized such that $A_E=A_F$ and $\Aaut(A_E,\tr_E)=\Aaut(A_E,\tr_F)$. 
\end{enumerate}
\end{rem}

According to the properties of the trace, we have the following result:

\begin{prop}
 Let $E, P \in \underset{\lambda=1}{\overset{n}{\bigoplus}} GL_{d_{\lambda}}(\C)$ be some multimatrices. Then $A_E=A_{PEP^{-1}}$ and we have a Hopf algebra isomorphism $$\Aaut(A_E, \tr_E) \simeq \Aaut(A_E, \tr_{PEP^{-1}})$$
\end{prop}

\begin{proof}
 The first assertion is obvious, and the rest follows from the universal property of the quantum automorphism group, with respect to the base change induced by the linear map $M\mapsto P^tMP^{-1t}$ and the fact that $\tr_E(P^tMP^{-1t})=\tr_{PEP^{-1}}(M)$.
\end{proof}

\section{$SO(3)$-deformation: the compact case}

This section is devoted to the proof of Theorem \ref{ClassTh} which classifies compact $SO(3)$-deformations. $\ $ \\ Let us describe the fusion semi-ring $\R^+(\gO(SO(3)))$: there exists a family of non-isomorphic simple comodules $(W_n)_{n\in \N}$ such that: $$W_0=\C, \quad W_n\otimes W_1 \simeq W_1 \otimes W_n \simeq W_{n-1} \oplus W_n \oplus W_{n+1}, \quad \dim (W_n) = 2n+1, \ \forall n \in \N^*$$
We aim to prove the following proposition:

\begin{prop}\label{PropAlgebre}
 Let $H$ be a compact $SO(3)$-deformation with simple comodules $(W_n)_{n\in \N}$ as above. Put $A=\C \oplus W_1$. Then there exist $H$-colinear maps 
\begin{align*}
 &A\otimes A \to A, &\C \to A,  & &\varphi : A \to \C
\end{align*}
making $(A,\varphi)$ into a measured $H$-comodule algebra. 

Moreover, there exists an antilinear map $*:A\to A$ making $A$ into a $*$-algebra and such that 
\begin{enumerate}
 \item $\varphi$ is positive, so that $A$ is a $C^*$-algebra,
 \item $(A,\varphi)$ is a normalizable measured $C^*$-algebra,
 \item $(A,\varphi)$ is a measured $H$-comodule $*$-algebra
\end{enumerate}
\end{prop}

After this paper was written, T. Banica informed us that Grossman and Snyder proved a related result (Theorem 3.4) in \cite{GrossSny}, working in arbitrary tensor $C^*$-categories with duals. More precisely, the first part of Proposition \ref{PropAlgebre} is a special case of Theorem 3.4 in \cite{GrossSny}. It is probably possible, although not immediate, to recover the full structure described in \mbox{Proposition \ref{PropAlgebre}} (\mbox{$*$-involution} and positivity of $\varphi$) from Theorem 3.4 in \cite{GrossSny}. Our independant proof is more concrete, and also has the merit that it brings some information in the non-compact case, see section \ref{GenCase}. We thank T. Banica for informing us about the paper \cite{GrossSny}.

Here, the proof of Proposition \ref{PropAlgebre} is the consequence of two lemmas. The different proofs being slightly technical, it seems useful to describe the example of $\gO(SO(3))\simeq \Aaut(M_2(\C), \tr)$, following Proposition 3.2 in \cite{BichNat}, which motivate the Lemmas \ref{prop1} and \ref{LemmInvol} below. 

At first, the reader can skip the proofs of Lemmas \ref{prop1} and \ref{LemmInvol} and go to the end of this section to see the construction leading to Proposition \ref{PropAlgebre} and the proof of Theorem \ref{ClassTh}.

\begin{ex}
Consider $(A,\varphi)=(M_2(\C), \tr)$ and the linear basis of $A$ consisting of the unit quaternions
\begin{equation*} e_0 = \begin{pmatrix}1 & 0\\0 & 1\end{pmatrix},
 e_1 = \begin{pmatrix}-i & 0\\0 & i\end{pmatrix},
e_2 = \begin{pmatrix}0 & -1\\1 & 0\end{pmatrix}, e_3 =
\begin{pmatrix}0 & i\\ i & 0\end{pmatrix}.
\end{equation*} 
These satisfy the following multiplication rules:
\begin{equation*}\label{quater}e_k^2 = -e_0, \; 1\leq k \leq 3, \quad e_1e_2 = e_3, \quad e_2e_3 = e_1,
\quad e_3e_1 = e_2. \end{equation*} 

We introduce some notations: for $1\leq k \neq l \leq 3$, $\langle kl\rangle \in \{ 1, 2, 3 \}$ is such that $\{ k, l \} \cup \{\langle kl\rangle\} = \{ 1, 2, 3 \}$, and let $\varepsilon_{kl} \in \{ \pm 1 \}$ be such that $e_k e_l=\varepsilon_{kl}e_{\langle kl\rangle}$. In particular, $\varepsilon_{kl}=-\varepsilon_{lk}$.

$\{ e_1, e_2, e_3 \}$ is a basis of ker$(\tr)$ which can be identified with the simple comodule $W=W_1$ in $\R^+(\gO(SO(3)))$, and we have the decomposition $$M_2(\C)= \C.e_0 \oplus W$$

Define the following colinear maps:

$$\begin{aligned}
  & e : W\otimes W \to \C,   & & e(e_k\otimes e_l)=-2 \ \delta_{kl}\\
  & e^* : \C \to W\otimes W, & & e^*(1)=-\frac{1}{2}\sum_{k=1}^3 e_k\otimes e_k\\
  & C : W\otimes W \to W,   & & C(e_k\otimes e_l)=(1-\delta_{kl})\varepsilon_{kl}e_{\langle kl\rangle}\\
  & D : W \to W\otimes W, & & D(e_k)= \sum_{p\neq k} \varepsilon_{kp}e_{\langle kp\rangle}\otimes e_p
 \end{aligned}$$

This maps satisfy some (compatibility) relations which are described in the following Lemma \ref{prop1}, with $\tau=3$ and $R=1$, and the multiplication in $M_2(\C)= \C.e_0 \oplus W$ decomposes into $$m(A\otimes B)=\frac{1}{2} e(A\otimes B) e_0 \oplus C(A\otimes B) \ \ \ \forall \ A,B \in W$$

\end{ex}
The rigidity provided by Schur's lemma and the fusion rules of $W\otimes W$ will allow us to see that this situation essentially holds in the general case.

We begin by a lemma: 

\begin{lemma}\label{LemmeTau}
 Let $F\in M_n(\C)$, $n\geq 3$, be such that $F\overline{F}=\pm I_n$. Then $\tr(FF^*)>2$.
\end{lemma}

\begin{proof}
First assume that $F\overline{F}=I_n$. Then according to \cite{BichonVaes} p.724, there exists a unitary matrix $U\in M_n(\C)$ and some real numbers $0 < \lambda_1 \leq \cdots \leq \lambda_k < 1$ such that 
\begin{equation*}
U^t F U = \begin{pmatrix} 0 & D(\lambda_1,\ldots,\lambda_k) & 0 \\ D(\lambda_1,\ldots,\lambda_k)^{-1} & 0 & 0 \\ 0 & 0 & I_{n-2k} \end{pmatrix} \; .
\end{equation*}
where $D(\lambda_1,\ldots,\lambda_k)$ denotes the diagonal matrix with the $\lambda_i$ along the diagonal. In that case, $$\tr(FF^*)=\sum_{i=1}^k (\lambda_i^2+\lambda_i^{-2}) + n-2k > 2$$

Now assume that $F\overline{F}=- I_n$. Then according to \cite{BichonVaes} p.724, $4\leq n$ is even and there exists a unitary matrix $u\in U_n(\C)$ and some real numbers $0 < \lambda_1 \leq \cdots \leq \lambda_{n/2} \leq 1$ such that
 $$U^t F U=\begin{pmatrix} 0 & D(\lambda_1,\ldots,\lambda_{n/2}) \\  -D(\lambda_1,\ldots,\lambda_{n/2})^{-1} & 0 \end{pmatrix}$$
In that case, $$\tr(F^*F)=\sum_{i=1}^{n/2} \lambda_i^{-2} + \lambda_i^2 > 2$$
\end{proof}

\begin{lemma}\label{prop1}
 Let $H$ be a compact $SO(3)$-deformation, with fundamental comodule $(W,\alpha)$ endowed with an $H$-invariant inner product. Then there exist morphisms of $H$-comodules
\begin{equation}
    \label{basic}
\begin{aligned}
  & e : W\otimes W \to \C \ \ \ & C : W\otimes W \to W \\
 \end{aligned}
   \end{equation}
and some scalars $\tau > 2$, $R\in \{\pm 1\}$ such that the following compatibility relations hold (where $e^*:\C \to W\otimes W$ is the adjoint of $e$ and $D:=(\id_W\otimes C)(e^* \otimes  \id_W) : W \to W\otimes W$):
\begin{subequations}
\label{rel}
  \begin{align}
 (e\otimes \id_W)(\id_W \otimes e^*)&=R \id_W  &(\id_W \otimes e) (e^* \otimes \id_W)&=R\id_W \\
  CD&=\id_W    &ee^*&=\tau\id_\C\\
  Ce^*&=0 &eD&=0\\
  e(C\otimes \id_W)&=e(\id_W\otimes C)    &(\id_W\otimes C)(e^* \otimes\id_W)&=(C\otimes \id_W)(\id_W\otimes e^*) \\
  (\id_W\otimes D)e^*&=(D\otimes \id_W)e^*   &(\id_W\otimes e)(D\otimes \id_W)&=(e\otimes \id_W)(\id_W\otimes D) 
\end{align}
\begin{align}
  (\id_W\otimes C)(D\otimes \id_W)&=(C\otimes \id_W)( \id_W \otimes D)=R(R-\tau)^{-1}\id_{W^{\otimes 2}} +(\tau-R)^{-1}e^* e +DC \\
(\id_W\otimes D)D&=R(R-\tau)^{-1}(e^* \otimes \id_W) +R(\tau-R)^{-1}(\id_W \otimes e^*) +(D\otimes \id_W)D \\
C(\id_W\otimes C)&=(R-\tau)^{-1}(\id_W\otimes e) +(\tau-R)^{-1}(e \otimes \id_W) +C(C \otimes \id_W)
\end{align}
 \end{subequations}
\end{lemma}

\begin{proof}
Let $(w_i)_{1\leq i \leq n}$ be an orthonormal basis of $W$ and let $(x_{ij})_{1\leq i,j \leq n}$ be the associated unitary multiplicative matrix. Recall that we write $\overline{x}=(x^*_{ij})$. From the fusion rules, we get $\dim W \geq 3$.

 We have $\overline{W} \simeq W$ by the fusion rules, hence there exist $F\in GL_n(\C)$ and $R\in \mathbb R^*$ such that $\bar x= F^{-1}xF$ and $F\overline{F}=R I_n$. Up to a nonzero real number, we can assume that $R\in \{\pm 1\}$. The map $e$ defined by $$e(w_i\otimes w_j)=\overline{F}_{ji}$$ is $H$-colinear and we have $$e^*(1)=\sum F_{ji} w_i \otimes w_j$$ and $e, e^*$ satisfy (\ref{rel}ab) for $\tau=\tr(F^*F) > 2$ according to Lemma \ref{LemmeTau}.

The fusion rules of $SO(3)$ give: $$W\otimes W \simeq \C \oplus W \oplus W^H_2 \ \ \ \ \ \emph{(FR)}$$
and there exists a nonzero (hence surjective) $H$-colinear map $C:W\otimes W \to W$. By Frobenius reciprocity, there exist isomorphisms

$$\begin{aligned}
  &\begin{array}{rl}
\Psi_1 : \Hom_H(W^{\otimes3}, \C) &\to \Hom_H(W^{\otimes 2}, W) \\
     f &\mapsto (\id_W \otimes f)(e^* \otimes \id_{W^{\otimes 2}})
\end{array}
&\begin{array}{rl}
\Psi_1^{-1} :  \Hom_H(W^{\otimes 2}, W) &\to \Hom_H(W^{\otimes3}, \C) \\
      g &\mapsto Re(\id_W\otimes g)
\end{array}
\end{aligned}$$

$$\begin{aligned} &\begin{array}{rl}
\Psi_2 : \Hom_H(W^{\otimes3}, \C) &\to \Hom_H(W^{\otimes 2}, W) \\
     f &\mapsto (f \otimes \id_W)( \id_{W^{\otimes 2}}\otimes e^* )
\end{array} &\begin{array}{rl}
\Psi_2^{-1} : \Hom_H(W^{\otimes 2}, W) &\to \Hom_H(W^{\otimes3}, \C) \\
     g &\mapsto Re(g \otimes \id_W)
\end{array}\end{aligned}$$

$$\begin{aligned} &\begin{array}{rl}
\Phi_1 : \Hom_H (W^{\otimes 2}, W) &\to \Hom_H (W, W^{\otimes 2}) \\
     f &\mapsto (f\otimes \id_W)(\id_W\otimes e^*)
\end{array} &\begin{array}{rl}
\Phi_1^{-1} : \Hom_H (W^{\otimes 2}, W) &\to \Hom_H (W, W^{\otimes 2}) \\
     g &\mapsto R(\id_W\otimes e)(g \otimes \id_W)
\end{array}\end{aligned}$$

$$\begin{aligned} &\begin{array}{rl}
\Phi_2 :\Hom_H (W^{\otimes 2}, W) &\to \Hom_H (W, W^{\otimes 2}) \\
     f &\mapsto (\id_W\otimes f)(e^* \otimes \id_W)
\end{array} &\begin{array}{rl}
\Phi_2^{-1} : \Hom_H (W^{\otimes 2}, W) &\to \Hom_H (W, W^{\otimes 2}) \\
     g &\mapsto R(e\otimes \id_W)(\id_W \otimes g)
\end{array}\end{aligned}$$
Put \begin{equation}\label{defD}
 D:=\Phi_2(C)=(\id_W\otimes C)(e^* \otimes \id_W)                                                                  
    \end{equation}
  By Schur's lemma, we can rescale $C$ such that $CD=\id_W$. Again by Schur's lemma, we have $Ce^* =0$ et $eD=0$. This gives relations (\ref{rel}bc).

Let us show that there exists $\omega \in \C^*$ such that the following relations hold:

\begin{equation}\label{EqD}
\begin{aligned}
    e(C\otimes \id_W)&= \omega e(\id_W\otimes C) \\  D=(\id_W\otimes C)(e^* \otimes  \id_W)&= \omega(C\otimes \id_W)(\id_W\otimes e^*)
\end{aligned}
  \end{equation}
According to Schur's lemma and the isomorphism $(FR)$, there exist $\omega_1$ and $\omega_2$ such that $$e(C\otimes \id_W)= \omega_1 e(\id_W,C)=\omega_1 R \Psi_1^{-1}(C) $$ and $$ (\id_W\otimes C)(e^* \otimes  \id_W)= \omega_2(C\otimes \id_W)(\id_W\otimes e^*)= \omega_2 \Phi_1(C)$$
Hence on the first hand we have
$$\begin{aligned}
\omega_1C&= R \Psi_1 (e(C\otimes \id_W)) \\
 &= R (\id_W\otimes e)(\id_W\otimes C \otimes \id_W)(e^* \otimes \id_W \otimes \id_W)  
 \end{aligned}$$
and on the other hand we have
$$\begin{aligned}
 \omega_2C &= \Phi_1^{-1}((\id_W\otimes C)(e^* \otimes \id_W)) \\
   &= R (\id_W\otimes e)(\id_W\otimes C \otimes \id_W)(e^* \otimes \id_W \otimes \id_W)
\end{aligned}$$
so $\omega_1=\omega_2:=\omega$. Since $C=\omega \Phi_2^{-1}((C\otimes\id_W)(\id_W\otimes e^*)) $, we have $\omega\neq 0$.

Let us show that 

\begin{equation}\label{EqC}
\begin{aligned}
 (\id_W\otimes e)(D\otimes \id_W)&= \omega(e\otimes \id_W)(\id_W\otimes D)=\omega RC \\  (\id_W\otimes D)e^* &=  \omega(D\otimes \id_W)e^*
\end{aligned}
\end{equation}
We have $D=(\id_W\otimes C)(e^* \otimes \id_W)=\omega (C\otimes \id_W)(\id_W\otimes e^*)$

$$
\begin{aligned}
(\id_W\otimes e)(D\otimes \id_W)&=(\id_W\otimes e)(\id_W\otimes C\otimes \id_W)(e^* \otimes \id_W\otimes \id_W) \\
&\overset{(\ref{EqD})}{=}\omega (\id_W\otimes e)(\id_W\otimes \id_W\otimes C)(e^* \otimes \id_W\otimes \id_W)\\
&\overset{(\ref{rel}a)}{=}\omega R C
\end{aligned}
 $$
and 
$$(e\otimes \id_W)(\id_W\otimes D)\overset{(\ref{defD})}{=}(e\otimes \id_W)(\id_W\otimes \id_W \otimes C)(\id_W\otimes e^* \otimes \id_W)\overset{(\ref{rel}a)}{=}R C$$
Hence $(\id_W\otimes e)(D\otimes \id_W)=\omega (e\otimes \id_W)(\id_W\otimes D)$. Moreover $$\begin{aligned}
           (\id_W\otimes D)e^* &= \omega (\id_W\otimes C \otimes \id_W) (\id_W\otimes \id_W \otimes e^*)e^*\\
&=\omega (\id_W\otimes C \otimes \id_W) (e^* \otimes \id_W \otimes \id_W)e^*\\
&\overset{(\ref{defD})}{=}\omega (D\otimes \id_W)e^*
          \end{aligned}$$

Let us show that \begin{equation}\label{HomW}     (\id_W\otimes C)(D\otimes \id_W)=\omega (C\otimes \id_W)( \id_W \otimes D) \end{equation}
Using relations (\ref{EqC}) and (\ref{rel}a), we compute $\omega^2(e\otimes \id_{W^{\otimes 2}})(\id_W\otimes D \otimes \id_W)(\id_W\otimes D)$ in two different ways: 

$$\begin{aligned}
\omega^2 (e\otimes \id_{W^{\otimes 2}})&(\id_W\otimes D \otimes \id_W)(\id_W\otimes D)\\
&\overset{(\ref{EqC})}{=}\omega(\id_W\otimes e \otimes \id_W)(D\otimes \id_W \otimes \id_W)(\id_W\otimes D)\\
&=\omega(\id_W\otimes e \otimes \id_W)(\id_W \otimes \id_W \otimes D)(D\otimes \id_W)\\
&=(\id_{W^{\otimes 2}} \otimes e)(\id_W\otimes D \otimes \id_W)(D\otimes \id_W)\\
&\overset{(\ref{defD})}{=}(\id_{W^{\otimes 2}} \otimes e)(\id_{W^{\otimes 2}}\otimes C \otimes \id_W)(\id_W\otimes e^* \otimes \id_{W^{\otimes 2}})(D\otimes \id_W)\\
&=\omega(\id_{W^{\otimes 2}} \otimes e)(\id_{W^{\otimes 3}}\otimes C)(\id_W\otimes e^* \otimes \id_{W^{\otimes 2}})(D\otimes \id_W)\\
&=\omega(\id_{W} \otimes \id_{W} \otimes e)(\id_W\otimes e^* \otimes \id_W) (\id_W\otimes C)(D\otimes \id_W)\\
&\overset{(\ref{rel}a)}{=}R\omega(\id_W\otimes C)(D\otimes \id_W)
  \end{aligned}$$

$$\begin{aligned}
\omega^2(e\otimes \id_{W^{\otimes 2}})&(\id_W\otimes D \otimes \id_W)(\id_W\otimes D)\\
&\overset{(\ref{defD})}{=}\omega^2(e\otimes \id_{W^{\otimes 2}})(\id_{W^{\otimes 2}}\otimes C \otimes \id_W)(\id_W\otimes e^* \otimes \id_{W^{\otimes 2}})(\id_W\otimes D)\\
&=\omega^2(C\otimes \id_W)(e\otimes \id_W \otimes \id_{W^{\otimes 2}})(\id_W\otimes e^* \otimes \id_{W^{\otimes 2}})(\id_W\otimes D)\\
&\overset{(\ref{rel}a)}{=}R\omega^2(C \otimes \id_W)(\id_W \otimes D)
  \end{aligned}$$
and hence $(\id_W\otimes C)(D\otimes \id_W)=\omega (C\otimes \id_W)( \id_W \otimes D)$.

Let us show that $\omega^3=1$. Denote $A:=\omega^2 e(C\otimes \id_W)(\id_W \otimes D)e^*$. On the first hand, we have: 
$$\begin{aligned}
 A&=\omega^2 e(C\otimes \id_W)(\id_W \otimes D)e^* \\
&\overset{(\ref{HomW})}{=} \omega e(\id_W\otimes C)(D \otimes \id_W)e^* \\
&\overset{(\ref{EqD})}{=}e(C\otimes \id_W)(D \otimes \id_W)e^* \\
&=e e^*
\end{aligned}$$
On the second hand, we have:

$$\begin{aligned}
A&=\omega^2 e(C\otimes \id_W)(\id_W \otimes D)e^*\\
&\overset{(\ref{EqC})}{=}\omega^3 e(C\otimes \id_W)(D \otimes \id_W)e^* \\
&=\omega^3 ee^*
\end{aligned}$$
Hence $\omega^3=1$.

By Frobenius reciprocity, we have isomorphisms

$$\begin{aligned}
\Phi : \End_H(W^{\otimes 2}) \to \Hom_H(W,W^{\otimes 3}), \ f \mapsto (\id_W\otimes f)(e^* \otimes \id_W )\\
\Psi : \End_H(W^{\otimes 2}) \to \Hom_H(W,W^{\otimes 3}), \ f \mapsto (f \otimes \id_W)(\id_W\otimes e^*)
\end{aligned}$$
Using relations (\ref{EqC}) and (\ref{rel}ab) we can compute the following:

$$\begin{aligned}
   \Phi((\id_W \otimes C)(D\otimes \id_W))&=\omega (D\otimes \id_W) D & \ \ \ \ \ \Psi((\id_W\otimes C)(D\otimes \id_W))&=\omega^2(\id_W\otimes D)D  \\
   \Phi(DC)&=(\id_W\otimes D)D   & \Psi(DC)&=\omega^2 (D \otimes \id_W)D\\
   \Phi(e^* e)&=R(\id_W\otimes e^*)  & \Psi(e^* e)&=R(e^* \otimes \id_W) \\
   \Phi(\id_{W^{\otimes 2}})&=(e^* \otimes \id_W) & \Psi(\id_{W^{\otimes 2}})&=(\id_W\otimes e^*)
  \end{aligned}$$

It is clear by $(FR)$ and Schur's lemma that $\{ \id_{W^{\otimes 2}},e^* e, DC \}$ is a basis of $\End_H(W\otimes W)$. Let $\alpha, \beta$ and $\gamma \in \C$ be such that
\begin{equation}\label{defB}
\begin{aligned}
  B&:=(\id_W\otimes C)(D\otimes \id_W)= \omega (C\otimes \id_W)( \id_W \otimes D)=\alpha \id_{W^{\otimes 2}} +\beta e^* e +\gamma DC
\end{aligned} 
\end{equation}
First, using relations (\ref{EqD}) and (\ref{rel}b), we have $eB=\omega^2 e=(\alpha + \tau \beta)e$ so $\alpha + \tau \beta=\omega^2$. We also have

\begin{equation}\label{phiB}
   \omega^2\Phi(B)= (D\otimes \id_W)D=\omega^2 \alpha  (e^* \otimes \id_W) + \omega^2 R \beta (\id_W \otimes e^*) + \omega^2 \gamma (\id_W \otimes D)D 
  \end{equation}
and
\begin{equation}\label{psiB}
 \begin{aligned}
 \Psi(B)&= \omega^2 (\id_W \otimes D)D =\alpha  (\id_W \otimes e^*) + R \beta (e^* \otimes \id_W) + \gamma \omega^2 (D \otimes \id_W)D 
  \end{aligned}
\end{equation}

which lead to the following relations between the coefficients: 

$$\left\{\begin{aligned}
\alpha + \tau \beta&=\omega^2\\
\alpha +\omega R \gamma \beta&=0\\
R \beta +\omega \gamma \alpha&=0\\
\gamma^2&=\omega  
 \end{aligned}\right.$$
In particular, $\alpha\neq 0\neq \beta$. Consider now $\omega^2C(C\otimes \id_W)\Phi(B)\in \End_H(W)$. On the first hand we have:
$$C(\id_W\otimes C)(D\otimes \id_W)D\overset{(\ref{defB})}{=}CBD=C(\alpha \id_{W^{\otimes 2}} +\beta e^* e +\gamma DC)D\overset{(\ref{rel}b)}{=}(\alpha+\gamma)\id_W $$
On the other hand, we have:
$$\begin{aligned}
   C(\id_W\otimes C)\big(&\omega^2 \alpha  (e^* \otimes \id_W) + \omega^2 R \beta (\id_W \otimes e^*) + \omega^2 \gamma (\id_W \otimes D)D\big) \\
  =&\omega^2 \alpha  C(1\otimes C)(e^* \otimes \id_W) + \omega^2 R \beta C(1\otimes C)(\id_W \otimes e^*) + \omega^2 \gamma C(1\otimes C)(\id_W \otimes D)D\\
  \overset{(\ref{defD})}{=}&\omega^2 \alpha  CD + \omega^2 \gamma \id_W \overset{(\ref{rel}b)}{=} \omega^2( \alpha + \gamma) \id_W.
  \end{aligned}$$
Hence $$ \alpha + \gamma=\omega^2( \alpha + \gamma)$$
By relation (\ref{HomW}), we have the identity
$$(\id_W\otimes B)(D\otimes \id_W)D=\omega (B\otimes \id_W)(\id_W\otimes D)D,$$
of which we develop the two sides:

$$\begin{aligned}
(\id_W\otimes B)&(D\otimes \id_W)D\\
\overset{(\ref{defB}, \ref{rel}a, \ref{EqD})}{=}&\alpha (D\otimes \id_W)D +\omega R \beta (\id_W\otimes e^*)+\gamma(\id_W\otimes D)(\id_W\otimes C)(D\otimes \id_W)D\\
\overset{(\ref{defB})}{=}&\alpha (D\otimes \id_W)D +\omega R \beta (\id_W\otimes e^*)+\gamma(\id_W\otimes D)(\alpha \id_{W^{\otimes 2}} + \beta e^* e + \gamma DC)D\\
\overset{(\ref{rel}bc)}{=}&\alpha (D\otimes \id_W)D +\omega R \beta (\id_W\otimes e^*)+(\alpha \gamma + \gamma^2)(\id_W\otimes D)D\\
\overset{(\ref{psiB})}{=}& \alpha (D\otimes \id_W)D +\omega R \beta(\id_W\otimes e^*)+(\alpha \gamma + \omega)(\omega \alpha  (\id_W \otimes e^*) \\
 &+ \omega R \beta (e^* \otimes \id_W) + \gamma (D \otimes \id_W)D)\\
=&(\alpha + \alpha \omega +\omega \gamma)(D\otimes \id_W)D+\omega(R \beta +\alpha^2\gamma  + \omega \alpha)(\id_W\otimes e^*)\\
&+\omega R \beta(\alpha \gamma + \omega)(e^* \otimes \id_W)
\end{aligned}$$

$$\begin{aligned}
\omega(B\otimes \id_W)&(\id_W\otimes D)D& &\\
\overset{(\ref{defB}, \ref{rel}a)}{=}&\omega \alpha (\id_W\otimes D)D+\omega R \beta (e^* \otimes \id_W)+\omega\gamma(D\otimes \id_W)(C\otimes \id_W)(\id_W\otimes D)D\\
\overset{(\ref{HomW})}{=}&\omega \alpha (\id_W\otimes D)D+\omega R \beta (e^* \otimes \id_W)+\gamma(D\otimes \id_W)(\id_W\otimes C)(D\otimes \id_W)D\\
\overset{(\ref{defB})}{=}&\omega \alpha (\id_W\otimes D)D+\omega R \beta (e^* \otimes \id_W)+\gamma(D\otimes \id_W)(\alpha \id_{W^{\otimes 2}} + \beta e^* e + \gamma DC)D\\
=&\omega \alpha (\id_W\otimes D)D+\omega R \beta (e^* \otimes \id_W)+(\alpha \gamma+ \gamma^2)(D\otimes \id_W)D\\
\overset{(\ref{psiB})}{=}&\omega \alpha (\omega \alpha  (\id_W \otimes e^*)+\omega R \beta (e^* \otimes \id_W) + \gamma (D \otimes \id_W)D)\\
&+\omega R \beta (e^* \otimes \id_W)+(\alpha \gamma+ \gamma^2)(D\otimes \id_W)D\\
=&\omega^2\alpha^2(\id_W\otimes e^*)+ \omega R \beta (\omega \alpha +1) (e^*\otimes \id_W)+(\omega \alpha \gamma+\alpha \gamma+\gamma^2)(D\otimes \id_W)D
\end{aligned}$$
This leads to several relations between the coefficients. In particular, we collect:

$$\left\{\begin{aligned}
&\alpha+\gamma=\omega^2(\alpha+\gamma)\\
&\alpha + \omega \alpha  +\omega \gamma=\omega \alpha \gamma+\alpha \gamma+\gamma^2\\
&\alpha \gamma + \omega=\omega \alpha +1\\
&\gamma^2=\omega  
 \end{aligned}\right.
$$

Assume that $\omega^2\neq 1$, then:

$$\left\{\begin{aligned}
&\alpha=-\gamma\\
&-\alpha^3+1=0\\
&\gamma^2=\omega  
 \end{aligned}\right.
$$

To summarize, we have 
$$\left\{\begin{aligned}
&\alpha + \tau \beta=\omega^2\\
&\alpha +\omega R \gamma \beta=0\\
&\alpha=-\gamma\\
&\alpha^3+1=0\\
&\gamma^2=\omega  
 \end{aligned}\right.
\Rightarrow
 \left\{\begin{aligned}
&\alpha^3+1=0\\
&\alpha + \tau \beta=\omega^2\\
&\alpha=-\gamma\\
&\alpha + R \beta=0\\
&\omega=\gamma^{-1} 
 \end{aligned}\right.$$
In particular, we have $$\alpha + \tau \beta=\omega^2 \Rightarrow \alpha -\tau R\alpha = -\alpha \Rightarrow \alpha(2-R\tau)=0 $$
Thus $\tau=2R$, which contradicts Lemma \ref{LemmeTau}. Hence $\omega^2=1=\omega^3$ and $\omega=1$. Moreover, we can consider once more the equality 
$$\alpha \gamma + \omega=\omega \alpha +1 \overset{\omega=1}{\Rightarrow} \alpha \gamma + 1= \alpha +1 $$
and we have $\gamma=1$.

Hence, we have
$$\begin{aligned}
\gamma&=\omega=1,\ \tau \neq R\\
\alpha&=-R\beta=R(R-\tau)^{-1}
\end{aligned}$$
and, in view of \ref{defB}, we have $$B=R(R- \tau)^{-1} \id_{W^{\otimes 2}} -(R- \tau)^{-1} e^* e +  DC$$
which gives relations (\ref{rel}f), and from relations (\ref{psiB}), we get relation (\ref{rel}g).
Finally, we have an isomorphism $$  \begin{array}{rl}
 \Omega : \End_H(W^{\otimes2}) &\longrightarrow \Hom_H(W^{\otimes 3}, W) \\
     f &\longmapsto R (\id_W \otimes e)(f \otimes \id_W)
\end{array} $$
In particular, using relations (\ref{rel}abc) and (\ref{EqD}), we can compute the following:

$$\begin{aligned}
   \Omega((\id_W \otimes C)(D\otimes \id_W))&= C(\id_W \otimes C)   \\
   \Omega(DC)&=C(C\otimes \id_W)   \\
   \Omega(e^* e)&=(e \otimes \id_W)  \\
   \Omega(\id_{W^{\otimes 2}})&=R (\id_W \otimes e) 
  \end{aligned}$$
This isomorphism applied to the relation (\ref{rel}f) gives the relation (\ref{rel}h).
\end{proof}

The next lemma will allow us to define the $*$-structure on the algebra $(A, \varphi)$ in Theorem \ref{ClassTh}.

\begin{lemma}\label{LemmInvol}
 Let $H$ be a compact $SO(3)$-deformation, with fundamental comodule $(W,\alpha)$. Then $R=1$ and there exist an antilinear map $*:W \to W$ such that:
\begin{subequations}\label{RelationInvol}
\begin{align}
   & w^{**}= w, & \forall  v\in W,\\
   & e(v^* \otimes w^*)= \overline{e(w\otimes v)}, & \forall  v,w\in W, \\
   & e(w \otimes w^*) > 0, & \forall w \in W\backslash(0),\\
   & C(v^* \otimes w^*)=C(w \otimes v)^*, & \forall  v,w\in W.
\end{align}
  \end{subequations}
\end{lemma}

\begin{proof}
Let $(w_i)_{1\leq i \leq n}$ be an orthonormal basis of $W$ and let $(x_{ij})_{1\leq i,j \leq n}$ be the associated unitary multiplicative matrix of coefficients. According to the beginning of the previous proof, the generators $x_{ij}$ and $x_{ij}^*$ are linked by the relations $\overline{x}= F^{-1}xF$, $F\in GL_n(\C)$ satisfying $F\overline{F}= R I_n$, $R\in\{\pm 1\}$. Let $*: W \to W$ be the antilinear map defined by $w_i^*=\sum_k w_kF_{ki}$. Note that we have \begin{equation}\label{Etoile1}
 e^*(1)=\sum_{i,j=1}^{n} F_{ji} w_i \otimes w_j = \sum_{i=1}^n w_i \otimes w_i^*.                                                                                                                                                                                                                                                                                                                                                                                                                                                                                                                                                                                                                                                                                                                                                                                                                                                                                                                                                                    \end{equation}

For $\gamma \in \C^*$, denote $C_{\gamma}=\gamma C$. We begin to show that, with $\gamma\in \{1,\mathrm{i}\}$ if $R=1$ and $\gamma \in \{ 1 \pm \mathrm{i} \}$ if $R=-1$ (where $\mathrm{i}^2=-1$), the following relations occur:
\begin{subequations}\label{RelationInvolTemp}
\begin{align}
   & w^{**}= R w, & \forall  v\in W,\\
   & e(v^* \otimes w^*)= R \overline{e(w\otimes v)}, & \forall  v,w\in W, \\
   & e(w \otimes w^*) > 0, & \forall w \in W\backslash(0),\\
   & C_{\gamma}(v^* \otimes w^*)=C_{\gamma}(w \otimes v)^*, & \forall  v,w\in W.
\end{align}
  \end{subequations}

We have
$$w_i^{**}=(\sum_k w_kF_{ki})^*=\sum_k w_k^*\overline{F}_{ki}=\sum_{k,l}w_lF_{lk}\overline{F}_{ki}=R w_i$$ and relation (\ref{RelationInvolTemp}a) follows.

Let us check the second relation. On the first hand, we have $e(w_i\otimes w_j)=\overline{F}_{ji}$ by definition, and on the other hand, we get $$e(w_j^*\otimes w_i^*)= \sum_{k,l} F_{ki}F_{lj}\overline{F}_{kl}=RF_{ji}=R\overline{e(w_i\otimes w_j)}$$and relation (\ref{RelationInvolTemp}b) follows.

Relation (\ref{RelationInvolTemp}c) can be seen as follows. Let $w=\sum_{i=1}^n \lambda_i w_i \in W$, $w\neq0$, we compute 
$$\begin{aligned}
    e(w \otimes w^*)&=\sum_{i,j} \lambda_i \overline{\lambda_j} e(w_i \otimes w_j^*)=\sum_{i,j,k}\lambda_i \overline{\lambda_j} F_{kj}\overline{F}_{ki}\\
                    &=\sum_k (\sum_i \lambda_i \overline{F}_{ki})(\sum_j \overline{\lambda_j} F_{kj})=\sum_k (\sum_i \lambda_i \overline{F}_{ki})\overline{(\sum_i \lambda_i \overline{F}_{ki})} >0.
  \end{aligned}$$

To show relation (\ref{RelationInvolTemp}d), remark that we have, for all $1\leq i\leq n$,
\begin{equation}\label{StarColin}
\begin{aligned}
  \alpha(w_i^*) &=\sum_k \alpha(w_k)F_{ki}=\sum_{k,p} w_p\otimes x_{pk}F_{ki} \\
  &=\sum_{p} w_p\otimes (xF)_{pi}=\sum_{p} w_p\otimes (F\overline{x})_{pi} \\
  &=\sum_{k,p} w_p\otimes F_{pk}x^*_{ki}=\sum_{k}w_{k}^*\otimes x^*_{ki},
\end{aligned}
\end{equation}
Define the antilinear map $$\#:W\otimes W \to W\otimes W, v\otimes w \mapsto (v\otimes w )^{\#}:=w^*\otimes v^* $$ According to (\ref{RelationInvolTemp}a), it is an involution and we have, for all $1\leq i,j\leq n$ $$\alpha_{W\otimes W} ((w_i\otimes w_j)^{\#})=\sum_{k,l} (w_k \otimes w_l)^{\#}\otimes (x_{ki}x_{lj})^*.$$
Hence the map $$\tilde{C} : W\otimes W \to W, w \mapsto C(w^{\#})^*$$ is $H$-colinear, and by Schur's lemma, there exists $\lambda \in \C$ such that $\tilde{C}=\lambda C$. In the same way, define the colinear map $$\tilde{D}:W \to W\otimes W,w \mapsto D(w^*)^{\#}.$$ Using relations (\ref{rel}b) and (\ref{RelationInvolTemp}a), it is clear that $\tilde{C}\tilde{D}=R\id_W$. Moreover, we have $\tilde{D}=(\tilde{C}\otimes \id_W)(\id_W \otimes e^*)$. Indeed, for all $1\leq i\leq n$,
$$\begin{aligned}
D(w_i^*)^{\#}&\overset{(\ref{defD})}{=}\big( (\id_W \otimes C)(e^*\otimes \id_W)(w_i^*)\big)^{\#}\overset{(\ref{Etoile1})}{=} \big( \sum_{p} w_p \otimes C( w_p^* \otimes w_i^*)\big)^{\#}\\
 &= \sum_{p} C( w_p^* \otimes w_i^*)^* \otimes w_p^* = \sum_{p} \tilde{C}(w_i \otimes w_p) \otimes w_p^* \\
 &\overset{(\ref{Etoile1})}{=} (\tilde{C} \otimes \id_W)(\id_W \otimes e^*)(w_i).
  \end{aligned}$$
Hence, according to relation (\ref{rel}b), $\tilde{C}\tilde{D}=\lambda^2 \id_W$, and $\lambda^2=R$. Choose $\gamma \in \C^*$, with $\gamma\in \{1,\mathrm{i}\}$ if $R=1$, $\gamma\in\{1\pm \mathrm{i}\}$ if $R=-1$, such that $\gamma R \overline{\lambda} = \overline{\gamma}$, we have the claimed relation (\ref{RelationInvolTemp}d) $$C_{\gamma}(v^* \otimes w^*)=\gamma R \overline{\lambda}(C(w \otimes v))^*=\overline{\gamma}(C(w \otimes v))^*=C_{\gamma}(w \otimes v)^*, \ \ \ \forall \ v,w\in W.$$

Let us show that $R=1$. We assume that $R=-1$ and we use relation (\ref{RelationInvolTemp}d) with $\gamma \in \{ 1 \pm \mathrm{i} \}$. On the first hand, we have, for $v,w \in W$, $$C_{\gamma}(w \otimes v)^{**} \overset{(\ref{RelationInvolTemp}a)}{=} -C_{\gamma}(w \otimes v) $$ and on the other hand, we have $$C_{\gamma}(w \otimes v)^{**}\overset{(\ref{RelationInvolTemp}d)}{=}(C_{\gamma}(v^* \otimes w^*))^*\overset{(\ref{RelationInvolTemp}d)}{=} C_{\gamma}(w^{**} \otimes v^{**})\overset{(\ref{RelationInvolTemp}a)}{=}C_{\gamma}(w \otimes v).$$ Since $C_{\gamma} \neq 0$, this is a contradiction, and $R=1$.

So far, we have the relations (\ref{RelationInvol}abc). Let us show the remaining relation (\ref{RelationInvol}d). For all $v,w \in W$, we have $C(v\otimes w)^* = \lambda C(w^* \otimes v^*)$ with $\lambda \in \{\pm 1\}$, and we need to show that $\lambda=1$. For all $1 \leq i \leq n$ and all $v\in W$, using (\ref{RelationInvol}c), we have $$e\big( C(v \otimes w_i) \otimes C(v \otimes w_i)^* \big) \geq 0.$$ Hence, for all $v \in W\backslash (0)$, we have 
$$\begin{aligned}
   0&\leq \sum_{i=1}^n e\big( C(v \otimes w_i) \otimes C(v \otimes w_i)^* \big) = \lambda \sum_{i=1}^n  e\big( C(v \otimes w_i) \otimes C(w_i^* \otimes v^*) \big)\\
 &\overset{(\ref{rel}d)}{=} \lambda \sum_{i=1}^n  e\big( C(C \otimes \id_W) (v \otimes w_i \otimes w_i^*) \otimes v^*\big) = \lambda e\big( C(C \otimes \id_W) (v \otimes \sum_{i=1}^n (w_i \otimes w_i^*)) \otimes v^*\big) \\
 &\overset{(\ref{Etoile1})}{=}  \lambda e\big( C(C \otimes \id_W)(\id_W \otimes e^*) (v) \otimes v^*\big) \overset{(\ref{defD})}{=} \lambda e\big( CD (v) \otimes v^*\big) \\ 
 &\overset{(\ref{rel}b)}{=} \lambda e( v \otimes v^*)
  \end{aligned}$$
Since $e( v \otimes v^*)>0$, we have $\lambda>0$ hence $\lambda=1$, and we have the claimed relation (\ref{RelationInvol}d).
\end{proof}

We are now able to prove Proposition \ref{PropAlgebre}.

\begin{proof}[Proof of Proposition \ref{PropAlgebre}]
 Let $H$ be a compact Hopf algebra whose corepresentation semi-ring is isomorphic to that of $SO(3)$. We write $(W_n^H)_{n\in \N}$ its family of simple comodules and we denote by $W:=W_1^H$ its fundamental comodule, with $\alpha_W'$ the associated coaction. We use the notations introduced in the previous lemmas.

The first thing to do is to define a finite-dimensional measured $C^*$-algebra $(A,\varphi)$ together with an $H$-coaction. Let $A$ be the $H$-comodule $\C\oplus W$, $\dim(A) \geq 4$. Endow $A$ with the following maps: let $m: A \otimes A \to A$, $u : \C \to A$ and $\varphi : A \to \C$ be the $H$-colinear maps defined by
$$\begin{aligned}
 m(\lambda \otimes \mu) &= \lambda \mu,  & \forall \lambda,  \mu \in \C,\\
 m(\lambda \otimes v) &= m(v \otimes \lambda) = \lambda v,     &\forall \lambda \in \C, \ v \in W,\\
 m(v \otimes w) &= \big( (\tau-1)^{-1} e(v\otimes w), C(v \otimes w)\big), &\forall  v,w \in W,\\
 u(1)&=(1,0):=1_A,& \\
 \varphi(\lambda, v) &= \lambda, &\forall \lambda \in \C, \ v \in W, 
\end{aligned}$$
and let $* : A \to A$ be the antilinear map defined by:
$$(\lambda, v)^* = (\overline{\lambda}, v^*)$$ where $*:W\to W$ is the antilinear map defined in Lemma \ref{LemmInvol}.

\begin{itemize}
 \item $m$ is associative:
The only non-trivial part is to check the associativity on $W$. This is done as follows: $$\begin{aligned}
m(m\otimes \id_W)_{|W\otimes W \otimes W}&= \big( (\tau-1)^{-1} e(C\otimes \id_W), \ C(C\otimes \id_W)+ (\tau-1)^{-1} (e\otimes \id_W)\big) \\
&\overset{(\ref{rel}dh)}{=}\big((\tau-1)^{-1} e(\id_W\otimes C), \ C(\id_W\otimes C) + (\tau-1)^{-1} (\id_W\otimes e)\big) \\ &= m(\id_W\otimes m)_{|W\otimes W \otimes W}
\end{aligned}.$$
Now we simplify the notations by writing the product $m\big( (\lambda,v)\otimes (\mu,w)\big):=(\lambda,v)(\mu,w)$.
\item $u$ is a unit: This is clear.
 
\item $A$ is a $*$-algebra:
$*:A\to A$ is indeed an antilinear involution by Lemma \ref{LemmInvol}, and we have 
$$\begin{aligned}
 \big((\lambda,v)(\mu, w)\big)^*&=\big( (\lambda \mu + (\tau-1)^{-1}e(v\otimes w), \ \lambda w + \mu v + C(v\otimes w) \big)^*\\
 &=\big(\overline{\lambda \mu} + (\tau-1)^{-1}\overline{e(v\otimes w)}, \ \overline{\lambda} w^* + \overline{\mu} v^* + C(v\otimes w)^*\big)\\
 &\overset{(\ref{RelationInvol}bd)}{=} \big( \overline{\mu} \overline{\lambda}+ (\tau-1)^{-1}e(w^*\otimes v^*), \ \overline{\lambda} w^* + \overline{\mu} v^* + C(w^*\otimes v^*) \big)\\
 &=(\mu, w)^*(\lambda,v)^*.
\end{aligned}$$

\item $\varphi$ is a faithful state on $A$: we have $\varphi(1_A)=1$ by definition, and

$$\begin{aligned}
\varphi((\lambda, w)(\lambda, w)^*)&=\varphi\big( \lambda \overline{\lambda} + (\tau-1)^{-1}e(w\otimes w^*),\ \lambda w^* + \overline{\lambda} w + C(w\otimes w^*)\big) \\
&= |\lambda|^2 + (\tau-1)^{-1}e(w\otimes w^*) 
\end{aligned}$$
Hence according to Lemmas \ref{LemmeTau} and \ref{LemmInvol}, we have, for all $(\lambda, w)\in A$ $$\varphi((\lambda, w)(\lambda, w)^*) \geq 0$$ with equality if and only if $(\lambda, w)=0$.
\end{itemize}
Thus $A$ is a finite dimensional $*$-algebra having a faithful state, and is a $C^*$-algebra. By Lemma \ref{LemmInvol} and by construction of the structure maps, $A$ is a $H$-comodule $*$-algebra and $\varphi$ is equivariant, thus by universality, there exists a Hopf $*$-algebra morphism $f : \Aaut(A,\varphi) \to H$ such that $(\id_W\otimes f) \circ \alpha_A = \alpha_A'$.
Finally, $W=\ker (\varphi)$ is a $\Aaut(A,\varphi)$-subcomodule of $A$, and by definition of the coactions on $A$, we have $(\id_W\otimes f) \circ \alpha_W = \alpha_W'$.

Let us show that $\varphi$ is normalizable. The map $$\tilde{\varphi} : A \overset{\id_A \otimes\tilde{\delta}}{\to} A\otimes A\otimes A \overset{\id_A \otimes m }{\to} A\otimes A \overset{m}{\to} A \overset{\varphi}{\to} \C$$ is a $H$-colinear map. Using Schur's lemma, we have $\dim(\Hom_H(A,\C))=1$, hence there exists $c\in \C$ such that $\tilde{\varphi} = c \varphi$. Let us compute $\tilde{\varphi}(1_A)$. A basis of $A$ is given by $a_1=1_A$, $a_i=w_{i-1}$ for $i=2,\dots, n+1$. Then we have 
$$B_{ij}:=\varphi  (a_i  a_j)=\left\{ \begin{array}{cc}
1 &\text{ if } i=j=1\\ 0 &\text{ if } i=1\neq j \text{ or } i\neq 1 = j \\ (\tau-1)^{-1} \overline{F}_{j-1,i-1} &\text{ in the other cases}\end{array}\right.$$ where $F \in GL_n(\C)$ is given by $e(w_i\otimes w_j) = \overline{F}_{ji}$. Hence $\tilde{\delta}$ is given by $$\tilde{\delta}(1)= \sum_{i,j=1}^{n+1} B_{ij}^{-1}a_i \otimes a_j=1_A\otimes 1_A + (\tau-1) e^*(1).$$  Hence, using relation (\ref{rel}b), we have $\tilde{\varphi}(1_A)=\tau+1>3$, so $\tilde{\varphi} = (\tau+1) \varphi$. Hence $\varphi$ is homogeneous. Moreover, we have $\varphi(1_A)=1$, hence $\varphi$ is normalizable.
\end{proof}

We are now able to prove Theorem \ref{ClassTh}, that is, to show that any compact $SO(3)$-deformation is isomorphic to the quantum automorphism group of a finite-dimensional measured $C^*$-algebra.

\begin{proof}[Proof of Theorem \ref{ClassTh}]
According to Proposition \ref{PropAlgebre} and its proof, there exist a normalizable measured $C^*$-algebra $(A,\varphi)$ and a $*$-Hopf algebra morphism $f : \Aaut(A,\varphi) \to H$ such that $W\subset A$ is a $\Aaut(A,\varphi)$-subcomodule and $(\id_W\otimes f) \circ \alpha_W = \alpha_W'$, where $\alpha_W$ is the coaction on $W$ of $\Aaut(A,\varphi)$. According to \cite{Ban2, BanFussCat}, $\Aaut(A,\varphi)$ is a compact $SO(3)$-deformation and we write $(W_n^A)_{n\in \N}$ its family of simple comodules. Then we have $f_*(W_1^A)\simeq W_1^H$, and by induction, we have $f_*(W_n^A)\simeq W_n^H$ for all $n\in \mathbb{N}$, hence, by a standard semi-ring argument, $f$ is an isomorphism of $*$-Hopf algebras and $H\simeq \Aaut(A,\varphi)$.
\end{proof}

\section{Representation theory of quantum automorphism groups}

We now investigate the case where $\varphi$ is not necessarily positive, and the aim of this section is to prove Theorem \ref{SO3comod}.

We will construct equivalences of monoidal categories by using appropriate Hopf bi-Galois objects (see \cite{Sch1}). We will work in the convenient framework of cogroupoids (see \cite{Bi2}). 

\begin{defi}
A $\C$-\emph{cogroupoid} $C$ consists of:
\begin{itemize}
 \item A set of objects $ob(C)$.
 \item For any $X,Y \in ob(C)$, a $\C$-algebra $C(X,Y)$.
 \item For any $X,Y,Z \in ob(C)$, algebra morphisms $$ \Delta_{X,Y}^Z : C(X,Y) \to C(X,Z)\otimes C(Z,Y) \text{ and }  \varepsilon_X : C(X,X)\to \C$$ and linear maps  $$S_{X,Y} : C(X,Y) \to C(Y,X)$$ satisfying several compatibility diagrams: see \cite{Bi2}, the axioms are dual to the axioms defining a groupoid.
\end{itemize}
A cogroupoid $C$ is said to be \emph{connected} if $C(X,Y)$ is a nonzero algebra for any $X,Y \in ob(C)$.
\end{defi}

Let $E\in \underset{\lambda=1}{\overset{n_E}{\bigoplus}} GL_{d_{\lambda}^E}(\C)$ and $F\in \underset{\mu=1}{\overset{n_F}{\bigoplus}} GL_{d_{\mu}^F}(\C)$ be two multimatrices. We denote $d_E:=d_{n_E}^E$ and $d_F:=d_{n_F}^F$. The algebra $\A(E,F)$ is the universal algebra with generators $X_{kl,\mu}^{ij,\lambda}$ ($1\leq \lambda \leq n_E$, $1\leq i,j \leq d^E_{\lambda}$, $1\leq \mu \leq n_F$, $1\leq k,l\leq d^F_{\mu}$) submitted to the relations

\begin{align*}
 &\sum_{q=1}^{d^E_{\nu}} X_{ij,\lambda}^{rq,\nu}X_{kl,\mu}^{qs,\nu}=\delta_{\lambda \mu} \delta_{jk} X_{il,\mu}^{rs,\nu}, & & \sum_{\mu=1}^{n_F}\sum_{k=1}^{d^F_{\mu}} X_{kk,\mu}^{ij,\lambda}=\delta_{ij},\\
 &\sum_{\mu=1}^{n_E}\sum_{k,l=1}^{d^E_{\mu}}E^{-1}_{kl,\mu} X_{ij,\lambda}^{kl,\mu}=F^{-1}_{ij,\lambda}, & & \sum_{r,s=1}^{d^F_{\mu}} F_{rs,\mu} X_{kr,\mu}^{ip,\lambda}X_{sl,\mu}^{qj,\nu}=\delta_{\lambda \nu}E_{pq,\lambda}X^{ij,\lambda}_{kl,\mu}.
\end{align*}
It is clear that $\A(E,E) = \Aaut(A_E,\tr_E)$ as an algebra.

We have the following lemma:

\begin{lemma}
\begin{itemize}
 \item  For any multimatrices $E\in \underset{\lambda=1}{\overset{n_E}{\bigoplus}} GL_{d_{\lambda}^E}(\C)$, $F\in \underset{\mu=1}{\overset{n_F}{\bigoplus}} GL_{d_{\mu}^F}(\C)$ and $G\in \underset{\nu=1}{\overset{n_G}{\bigoplus}} GL_{d_{\nu}^G}(\C)$, there exist algebra maps 
$$\Delta_{E,F}^{G} : \mathcal{A}(E,F) \to \mathcal{A}(E,G) \otimes \mathcal{A}(G,F)$$
defined by $\Delta_{E,F}^{G} (X_{kl,\mu}^{ij,\lambda})=\underset{\nu=1}{\overset{n_G}{\sum}}\underset{r,s=1}{\overset{d_{\nu}^G}{\sum}} X_{rs,\nu}^{ij,\lambda} \otimes X_{kl, \mu}^{rs,\nu}$ ($1\leq \lambda \leq n_E$, $1\leq i,j \leq d^E_{\lambda}$, $1\leq \mu \leq n_F$, $1\leq k,l\leq d^F_{\mu}$) and $$\varepsilon_{E} : \mathcal{A}(E) \to \C $$ defined by $\varepsilon_E (X_{kl,\mu}^{ij,\lambda})= \delta_{ik}\delta_{jl}\delta_{\lambda \mu}$ ($1\leq \lambda, \mu \leq n_E$, $1\leq i,j \leq d^E_{\lambda}$, $1\leq k,l\leq d^E_{\mu}$) such that, for any multimatrix $M\in \underset{\eta=1}{\overset{n_M}{\bigoplus}} GL_{d_{\eta}^M}(\C)$, the following diagrams commute: 

$$\xymatrix{
\mathcal{A}(E,F) \ar[r]^{\Delta_{E,F}^{G} \ \ \ \ \ \ } \ar[d]_{\Delta_{E,F}^{M}} & \mathcal{A}(E,G)\otimes \mathcal{A}(G,F) \ar[d]^{\Delta_{E,G}^{M}\otimes \id} \\
\mathcal{A}(E,M)\otimes \mathcal{A}(M,F) \ar[r]_{\id\otimes \Delta_{M,F}^{G} \ \ \ \ \ \ } & \mathcal{A}(E,M)\otimes \mathcal{A}(M,G) \otimes \mathcal{A}(G,F)
}$$

$$\xymatrix{
\mathcal{A}(E,F)  \ar[d]_{\Delta_{E,F}^{F}} \ar[rd]& \\
\mathcal{A}(E,F)\otimes \mathcal{A}(F) \ar[r]_{\ \ \ \ \ \ \id\otimes \varepsilon_{F}} & \mathcal{A}(E,F)
} \ \ \ \
\xymatrix{
\mathcal{A}(E,F) \ar[d]_{\Delta_{E,F}^{E}} \ar[rd]& \\
\mathcal{A}(E)\otimes \mathcal{A}(E,F) \ar[r]_{\ \ \ \ \ \ \varepsilon_{E}\otimes \id} & \mathcal{A}(E,F)
}$$
 \item For any multimatrices $E\in \underset{\lambda=1}{\overset{n_E}{\bigoplus}} GL_{d_{\lambda}^E}(\C)$, $F\in \underset{\mu=1}{\overset{n_F}{\bigoplus}} GL_{d_{\mu}^F}(\C)$, there exists an algebra morphism $$S_{E,F}:\mathcal{A}(E,F) \to \mathcal{A}(F,E)^{op}$$ defined by $S_{E,F}(X_{kl,\mu}^{ij,\lambda})= \underset{r=1}{\overset{d^E_{\lambda}}{\sum}}\underset{s=1}{\overset{d^F_{\mu}}{\sum}} E_{jr,\lambda} F^{-1}_{sl,\mu} X_{ri,\lambda}^{sk, \mu}$ ($1\leq \lambda \leq n_E$, $1\leq i,j \leq d^E_{\lambda}$, $1\leq \mu \leq n_F$, $1\leq k,l\leq d^F_{\mu}$) such that the following diagrams commute: 

$$\xymatrix{
\mathcal{A}(E) \ar[r]^{\varepsilon_{E}} \ar[d]_{\Delta_{E,E}^{F}}& \C \ar[r]^{u \ \ \ \ \ }& \mathcal{A}(E,F)  \\
\mathcal{A}(E,F)\otimes \mathcal{A}(F,E) \ar[rr]_{\id\otimes S_{F,E}} & & \mathcal{A}(E,F)\otimes \mathcal{A}(E,F)\ar[u]^{m} 
}$$ 
$$\xymatrix{
\mathcal{A}(E) \ar[r]^{\varepsilon_{E}} \ar[d]_{\Delta_{E,E}^{F}}& \C \ar[r]^{u \ \ \ \ \ }& \mathcal{A}(F,E)  \\
\mathcal{A}(E,F)\otimes \mathcal{A}(F,E) \ar[rr]_{S_{E,F}\otimes \id} & & \mathcal{A}(F,E)\otimes \mathcal{A}(F,E)\ar[u]^{m} 
}$$

\end{itemize}
\end{lemma}

\begin{proof}
 The existence of the algebra morphisms is a consequence of the universal property of $\mathcal{A}(E,F)$, and the commutativity of the diagrams can easily be checked on the generators.
\end{proof}

The previous lemma allows us to define a cogroupoid in the following way:

\begin{defi}
 The cogroupoid $\A$ is defined as follows:
\begin{enumerate}
 \item $ob(\A)=\{E\in \underset{\lambda=1}{\overset{n}{\bigoplus}} GL_{d_{\lambda}^E}(\C); \ d_{E}>1 \},$
 \item For $E,F \in ob(\A)$, the algebra $\A(E,F)$ is the algebra defined above,
 \item The structural maps $\Delta_{\bullet, \bullet}^{\bullet}$, $\varepsilon_{\bullet}$, and $S_{\bullet, \bullet}$ are defined in the previous lemma.
\end{enumerate}
\end{defi}

\begin{rem}
 \begin{enumerate} 
\item The condition $d_{E}>1$ rules out the case of $\Aaut(C(X_n), \psi)$. This is discussed in the Appendix and a solution is provided by Theorem \ref{twist}.
\item The present construction is related to the bialgebras constructed by Tambara in \cite{Tambara}.
 \end{enumerate}
\end{rem}

We now need to study the connectedness of this cogroupoid. We begin by the following technical lemma (we refer to the Appendix for its proof):

\begin{lemma}\label{LemTech}
 Let $E$, $F\in ob(\A)$. Assume that $\Tr(E^{-1})=\Tr(F^{-1})$ and $\tr(E_{\lambda})=\tr(F_{\mu})$ for all $\lambda,\mu$. Then the algebra $\A(E,F)$ is nonzero.
\end{lemma}
In particular, we have the following corollary.

\begin{cor}
Let $\tau, \theta \in \C$. Let $\A^{\tau, \theta}$ be the full subcogroupoid of $\A$ with objects 
$$ ob(\A^{\tau, \theta})=\left\{
\begin{array}{c|c}
 & 1<d_E,\\
E \in \underset{\lambda=1}{\overset{n}{\bigoplus}} GL_{d_{\lambda}^E}(\C)  &  \Tr(E^{-1})=\tau,\\
 & \tr(E_{\lambda})=\theta, \ \forall  \lambda
\end{array}
 \right\}$$ Then $\A^{\tau, \theta}$ is connected.
\end{cor}

Using \cite{Bi2}, Proposition 2.8 and Schauenburg's Theorem 5.5 \cite{Sch1}, we have the following result.

\begin{cor}
 Let $E\in \underset{\lambda=1}{\overset{n_E}{\bigoplus}} GL_{d_{\lambda}^E}(\C)$, $F\in \underset{\mu=1}{\overset{n_F}{\bigoplus}} GL_{d_{\mu}^F}(\C)$ be two multimatrices such that $1<d_E,d_F$, $\Tr(E^{-1})=\Tr(F^{-1})$ and $\tr(E_{\lambda})=\tr(F_{\mu})$ for all $\lambda, \mu$.  Then we have a $\C$-linear equivalence of  monoidal categories $$\Com(\Aaut(A_E,\tr_E))\simeq^{\otimes}\Com(\Aaut(A_F,\tr_F))$$ between the comodule categories of $\Aaut(A_E,\tr_E)$ and $\Aaut(A_F,\tr_F)$ respectively.
\end{cor}

Moreover, we have the following twisting result, inspired by \cite{QG4points}.

\begin{theo}\label{twist}
 Let $n\in \N$. Then the Hopf algebras $\Aaut(\C^n \oplus \C^4)$ and $\Aaut(\C^n \oplus (M_2(\C), \tr))$ are 2-cocycle twists of each other. In particular, they have monoidal equivalent comodule categories.
\end{theo}
We only sketch the proof of this result by giving the principal ideas but without performing the computations. One may also invoke \cite{DeRVaVenn}, Theorem 4.7.

\begin{proof}
 The first step is to give a new presentation of these Hopf algebras by using a different basis for the associated measured algebras. In the case of $A=\C^n \oplus \C^4$, we use the linear basis given by the canonical basis on $\C^n$ and the particular basis given in \cite{QG4points} Theorem 3.1. on $\C^4$, and when $A=\C^n \oplus (M_2(\C), \tr)$, we use the canonical basis on $\C^n$ and the quaternionic basis used in \cite{BichNat} Proposition 3.2. on $(M_2(\C), \tr)$. 

The cocycle $\sigma$ is given by the composition of the non trivial 2-cocycle of the Klein group $V$ (linearly extended to $\C [V]$) and the Hopf algebra surjection (see \cite{QG4points} Theorem 5.1) $$\Aaut(\C^n \oplus (M_2(\C), \tr)) \to \Aaut(M_2(\C), \tr) \to \C [V]$$ 

The computations show the existence of a Hopf algebra morphism from $\Aaut(\C^n \oplus \C^4)$ to $\Aaut(\C^n \oplus (M_2(\C), \tr))^{\sigma}$ which is an isomorphism by Tannaka Krein reconstruction techniques.
\end{proof}

This result enables us to optimize the following result by including the quantum permutation group.

\begin{cor}\label{EquComod}
 Let $(A_E, \tr_E)$ be a finite dimensional, semisimple, measured algebra of dimension $\dim A_E \geq 4$, where $E\in \underset{\lambda=1}{\overset{n}{\bigoplus}} GL_{d_{\lambda}^E}(\C)$ is a normalizable multimatrix. Then there exist $q\in \C^*$ and a $\C$-linear equivalence of  monoidal categories $$\Com(\Aaut(A_E,\tr_E))\simeq^{\otimes}\Com(\gO(SO_{q^{1/2}}(3)))$$ between the comodule categories of $\Aaut(A_E,\tr_E)$ and $\gO(SO_{q^{1/2}}(3))$ respectively. If $E$ is normalized, $q\in \C^*$ satisfies $q^2-\Tr(E^{-1})q+1=0$.
\end{cor}

\begin{proof}
 First assume that $1<d_E$. According to Remark \ref{RemNorm}, there exists a normalized multimatrix $F \in \underset{\lambda=1}{\overset{n}{\bigoplus}} GL_{d_{\lambda}^E}(\C)$ such that $\Aaut(A_E,\tr_E)= \Aaut(A_F,\tr_F)$ as Hopf algebras. Choose $q\in \C^*$ such that $\Tr(F^{-1})=q+q^{-1}=\tr(F_{\lambda})$ for $\lambda=1,\dots, n$. According to the previous corollary, we have a $\C$-linear equivalence of  monoidal categories $$\Com(\Aaut(A_F,\tr_F))\simeq^{\otimes}\Com(\Aaut(M_2(\C),\tr_q)).$$ Hence according to Example \ref{ExSO3} (\ref{ExSOq3}) we have a $\C$-linear equivalence of monoidal categories $$\Com(\Aaut(A_E,\tr_E))\simeq^{\otimes}\Com(\gO(SO_{q^{1/2}}(3))).$$

 If $E=(e,\dots,e)\in (\C^*)^m$, then $A_E=\C^m=\C^n \oplus \C^4$ with $n\in \N$ by assumption. Using Theorem \ref{twist}, we have a monoidal equivalence $$\Com (\Aaut(A_E,\tr_E)) \simeq^{\otimes} \Com(\Aaut(\C^n \oplus (M_2(\C),\tr)))$$ and we can apply the previous reasoning.
\end{proof}

In particular, Theorem \ref{SO3comod} is a consequence of Corollary \ref{EquComod}.

\section{$SO(3)$-deformations: the general case}\label{GenCase}

We would like to say a word about the $SO(3)$-deformations in the general case. Unlike in the compact case, we have not been able in general to associate a measured algebra $(A,\varphi)$ to an arbitrary $SO(3)$-deformation. This situation occurs because of the lack of analog of Lemma \ref{LemmeTau} in the general case. However, it is possible to give some partial results and directions concerning the general classification problem.

\subsection{The representation theory of $SO_{q^{1/2}}(3)$}

Recalling that $\gO(SO_{q^{1/2}}(3))$ is a Hopf subalgebra of $\gO(SL_q(2))$, it is possible to describe its corepresentation semi-ring, as follows:

\begin{theo}\label{RacUnit}
 Let $q\in \C^*$. We say that $q$ is generic if $q$ is not a root of unity or if $q\in\{\pm 1\}$. If $q$ is not generic, let $N\geq 3$ be the order of $q$, and put \begin{equation*}
N_0 = 
\begin{cases} N & \text{if $N$ is odd},\\
N/2 & \text{if $N$ is even}.
\end{cases}
\end{equation*}

\begin{itemize}
 \item First assume that $q$ is generic. Then $\gO(SO_{q^{1/2}}(3))$ is cosemisimple and has a family of non-isomorphic simple comodules $(W_n)_{n\in \N}$ such that: $$W_0=\C, \quad W_n\otimes W_1 \simeq W_1 \otimes W_n \simeq W_{n-1} \oplus W_n \oplus W_{n+1}, \quad \dim (W_n) = 2n+1, \ \forall n \in \N^*.$$

Furthermore, any simple $\gO(SO_{q^{1/2}}(3))$-comodule is isomorphic to one of the comodule $W_n$.

\item Now assume that $q$ is not generic and that $N_0= 2N_1$, $N_1\in \mathbb{N}^*$. Then $\gO(SO_{q^{1/2}}(3))$ is not cosemisimple. There exist families $\{V_n, \ n \in \N\}$, $\{W_n, \ n=0,\dots, N_1-1\}$ of non-isomorphic simple comodules (except for $n=0$ where $V_0 = W_0 = \C$), such that
$$V_n \otimes V_1 \simeq V_1\otimes V_n \simeq V_{n-1} \oplus V_{n+1}, \quad \dim V_n = n+1, \ \forall n \in \N^*.$$
$$W_n \otimes W_1 \simeq W_1 \otimes W_n \simeq W_{n-1}\oplus W_n \oplus W_{n+1}, \quad \dim W_n = 2n+1, \ \forall n=1,\dots, N_1-1.$$

The comodule $W_{N_1-1}\otimes W_1$ is not semisimple. It has a simple filtration $$(0) \subset W_{N_1-2} \oplus W_{N_1-1} \subset  Y \subset W_{N_1-1}\otimes W_1$$ with $W_{N_1-1}\otimes W_1/Y\simeq W_{N_1-1}$ and $Y/( W_{N_1-2} \oplus W_{N_1-1}) \simeq V_1$. 

The comodules $W_n \otimes V_m\simeq V_m \otimes W_n$ ($m\in \N$ and $n=0,\dots, N_1-1$) are simple and any simple $\gO(SO_{q^{1/2}}(3))$-comodule is isomorphic with one of these comodules.

\item Finally assume that $q$ is not generic and that $N_0=2N_1-1$, $N_1\in \mathbb{N}^*$. Then $\gO(SO_{q^{1/2}}(3))$ is not cosemisimple. There exist families $\{V_n, \ n \in \N\}$, $\{U_n, \ n=0,\dots, N_0-1 \}$ of vector spaces (with dimension $\dim V_n =\dim U_n =n+1$) such that the families $\{V_{2n}, \ n \in \N\}$ $\{U_{2n}, \ n=0, \dots, N_1-1\}$ and $\{V_{2n+1} \otimes U_{2m+1}, \ n \in \N, \ m=0,\dots, N_1-1\}$ are non-isomorphic simple $\gO(SO_{q^{1/2}}(3))$-comodules (except for $n=0$ where $V_0 = U_0 =  \C$). They satisfy the fusion rules induced by 
$$V_n\otimes V_1 \simeq V_1\otimes V_n \simeq V_{n-1}\oplus V_{n+1}, \ \forall n\in \N^* $$
$$U_n\otimes U_1 \simeq U_1 \otimes U_n \simeq U_{n-1} \oplus U_{n+1} \ \forall n=1,\dots, N_0-1.$$

The comodule $U_{2(N_1-1)} \otimes U_2$ is not simple. It has a simple filtration $$(0) \subset U_{2(N_1-2)} \subset Y \subset U_{2(N_1-1)} \otimes U_2 $$ where $U_{2(N_1-1)} \otimes U_2/Y \simeq U_{2(N_1-2)}$ and $Y/U_{2(N_1-2)} \simeq U_1\otimes V_1$. The comodules $V_n\otimes U_m\simeq U_m \otimes V_n$ (with $n\equiv m (\rm{mod} 2)$) are simple, and any simple $\gO(SO_{q^{1/2}}(3))$-comodule is isomorphic with one of these comodules. 
\end{itemize}

\end{theo}

\begin{proof}
 We first collect some facts about $SL_q(2)$, $SO_{q^{1/2}}(3)$ and Hopf subalgebras. See \cite{KS} for the relations between $SL_q(2)$ and $SO_{q^{1/2}}(3)$ and \cite{KP} for the corepresentation theory of $SL_q(2)$.

\begin{itemize}
 \item Let $a,b,c,d$ be the matrix coefficients of the fundamental 2-dimensional $\gO(SL_q(2))$-comodule. Then $\gO(SO_{q^{1/2}}(3))$ is isomorphic to the Hopf subalgebra of $\gO(SL_q(2))$ generated by the even degree monomials in $a,b,c,d$. Moreover, we have a Hopf algebra isomorphism $\gO(SL_q(2))^{\rm{Co} \C[\Z_2]}\simeq \gO(SO_{q^{1/2}}(3)) $.
 \item When $q$ is not generic, the matrix $v=(v_{ij})_{1\leq  i,j\leq 4}$ with $v_{11}=a^{N_0}$, $v_{12}=b^{N_0}$, $v_{21}=c^{N_0}$ and $v_{22}=d^{N_0}$ is multiplicative, associated to the $\gO(SL_q(2))$-comodule $V_1$.
 \item Let $A\subset B$ be a Hopf algebra inclusion. Then an $A$-comodule is semisimple if and only if it is semisimple as a $B$-comodule. In particular, if $B$ is cosemisimple, so is $A$.
\end{itemize}

From those facts, we deduce that the $\gO(SO_{q^{1/2}}(3))$-comodules are exactly the $\gO(SL_q(2))$-comodules with matrix coefficients of even degree in $a,b,c,d$. The end of the proof comes from combining this with the results and proof from \cite{KP}.
\end{proof}

%

\subsection{The general case}

The study of the fusion rules of $SO(3)$ gives the following:

\begin{lemma}\label{prop1bis}
 Let $H$ be a $SO(3)$-deformation, with fundamental comodule $(W,\alpha)$. Then there exist morphisms of $H$-comodules
\begin{equation}
    \label{basicbis}
\begin{aligned}
  & e : W\otimes W \to \C \ \ \ & & \delta : \C \to W\otimes W \\
  & C : W\otimes W \to W \ \ \ & & D : W \to W\otimes W,
 \end{aligned}
   \end{equation}
a third root of unity $\omega \in \C$ and a unique nonzero scalar $\tau \in  \C^*$ satisfying the following compatibility relations:
\begin{subequations}\label{EquProp1bis}
  \begin{align}
 (e\otimes \id_W)(\id_W \otimes \delta)&=\id_W  &(\id_W \otimes e) (\delta \otimes \id_W)&=\id_W \\
 D&=(\id_W\otimes C)(\delta \otimes  \id_W) & & \\
  CD&=\id_W    &e\delta&=\tau\id_\C\\
  C\delta&=0 &eD&=0\\
  (\id_W\otimes C)(\delta \otimes\id_W)&=\omega (C\otimes \id_W)(\id_W\otimes \delta)   &e(C\otimes \id_W)&=\omega e(\id_W\otimes C) \\
  (\id_W\otimes e)(D\otimes \id_W)&=\omega (e\otimes \id_W)(\id_W\otimes D)   &(\id_W\otimes D)\delta&=\omega (D\otimes \id_W)\delta 
\end{align}
Moreover, if $\omega \neq 1$, we have $\tau = 2$, and if $\omega=1$, we have $\tau \neq 1$ and
\begin{align}
  (\id_W\otimes C)(D\otimes \id_W)&=(C\otimes \id_W)( \id_W \otimes D)=(1-\tau)^{-1}\id_{W^{\otimes 2}} +(\tau-1)^{-1}\delta e +DC \\
(\id_W\otimes D)D&=(1-\tau)^{-1}(\delta \otimes \id_W) +(\tau-1)^{-1}(\id_W \otimes \delta) +(D\otimes \id_W)D \\
C(\id_W\otimes C)&=(1-\tau)^{-1}(\id_W\otimes e) +(\tau-1)^{-1}(e \otimes \id_W) +C(C \otimes \id_W)
\end{align}
 \end{subequations}
\end{lemma}

\begin{proof}
The fusion rules for $SO(3)$ give: $$W\otimes W \simeq \C \oplus W \oplus W^H_2$$
Then there exist $H$-colinear maps $e, \delta$ and $C$ satisfying (\ref{EquProp1bis}a), and a scalar $\tau \in \C$ such that $e\delta=\tau\id_\C$. By cosemisimplicity, there exists $\delta'$ such that $e\delta'=\id_{\C}$ and by Schur's lemma, there exists $\alpha \in \C^*$ such that $\delta' = \alpha \delta$. Hence $\tau \neq 0$. Moreover, any rescaling of $e$ and $\delta$ that leaves (\ref{EquProp1bis}a) intact also leaves $\tau$ invariant, hence $\tau$ only depends on $H$.

The rest of the proof follows the one of Lemma \ref{prop1} but without Lemma \ref{LemmeTau}.
\end{proof}

In the rest of this paper, it seems convenient to distinguish the $SO(3)$-deformations by whether or not $\omega=1$.

\begin{nota}
 Let $H$ be a $SO(3)$-deformation. We say that $H$ is of type $\mathbf{I_{\tau}}$ if $\omega=1$, where $\tau \in \C^*$ is determined by $H$ according to Lemma \ref{prop1bis}. Otherwise, we say that $H$ is of type $\mathbf{II}$ (in that case, we always have $\tau=2$).
\end{nota}

$SO(3)$-deformations of type $\mathbf{I_{\tau}}$ are close to the compact case:

\begin{prop}\label{prop2}
Let $H$ be a $SO(3)$-deformation of type $\mathbf{I_{\tau}}$. Then there exist a finite dimensional, semisimple, measured algebra $(A,\varphi)$ with $\dim A \geq 4$, and a Hopf algebra morphism $f : \Aaut(A,\varphi) \to H$ such that $W\subset A$ is a $\Aaut(A,\varphi)$-subcomodule and $(\id_W\otimes f) \circ \alpha_W = \alpha_W'$, where $\alpha_W$ et $\alpha_W'$ are the coactions on $W$ of $\Aaut(A,\varphi)$ and $H$ respectively. Moreover, if $\tau\neq -1$, we can assume $(A,\varphi)$ normalized.
\end{prop}

\begin{proof}
The construction is essentially the same as in the proof of Theorem \ref{ClassTh}. The only difference is about the semisimplicity of the algebra.

Let $A$ be the $H$-comodule $\C\oplus W$ with $\dim A \geq 4$. Endow $A$ with the following $H$-colinear maps: define a product and a unit by 
$$(\lambda,v)(\mu, w)=(\lambda \mu + (\tau-1)^{-1}e(v\otimes w), \lambda w + \mu v + C(v\otimes w)), \ \ 1_A = (1, 0)$$
and a measure $\varphi : A \to \C$ by $\varphi(\lambda, v)=  \lambda$. As in the proof of Theorem \ref{ClassTh} and using relations (\ref{EquProp1bis}ei), $(A,m,u,\varphi)$ is a finite dimensional measured algebra.

Consider $\tilde{\delta}: \C \to A \otimes A$ defined by $\tilde{\delta}(1)= 1_A\otimes 1_A + (\tau - 1)\delta(1)$. Let $w \in W \subset A$. Then in $A\otimes A$, we have
$$\begin{aligned}
   w \tilde{\delta}(1) &= w\otimes 1 + 1\otimes (e\otimes \id_W)(\id_W \otimes \delta)(w)  + (\tau - 1)(C\otimes \id_W)(\id_W \otimes \delta)(w) \\
&\overset{(\ref{EquProp1bis}a)}{=} w\otimes 1 + 1\otimes w  + (\tau - 1)(C\otimes \id_W)(\id_W \otimes \delta)(w)\\
&\overset{(\ref{EquProp1bis}e)}{=} 1\otimes w + w\otimes 1 + (\tau - 1)(\id_W\otimes C)(\delta \otimes \id_W)(w)\\
&\overset{(\ref{EquProp1bis}a)}{=} 1\otimes w + (\id_W\otimes e)(\delta \otimes \id_W)(w)\otimes 1+ (\tau - 1)(\id_W\otimes C)(\delta \otimes \id_W)(w)=\tilde{\delta}(1) w.
  \end{aligned}$$
Hence for all $a\in A$, we have $a\delta(1)=\delta(1)a \in A\otimes A$. Put $r := (\tau+1)^{-1} \tilde{\delta(1)}$ so that $m(r)=1_A$ and $$s : A \to A\otimes A, \ a \mapsto a r$$ In view of the previous facts, $s$ is a $A$-$A$-bimodule morphism, and $m\circ s = \id_A$, then $A$ is separable and is semisimple.

Then $(A,\varphi)$ is a finite dimensional, semisimple, measured algebra. By construction of the structure maps, $A$ is a $H$-comodule algebra and $\varphi$ is equivariant, thus by universality, there exists a Hopf algebra morphism $f : \Aaut(A,\varphi) \to H$ such that $(\id_W\otimes f) \circ \alpha_A = \alpha_A'$.
Finally, $W=\ker (\varphi)$ is a $\Aaut(A,\varphi)$-subcomodule of $A$, and by definition of the coactions on $A$, we have $(\id_W\otimes f) \circ \alpha_W = \alpha_W'$.

Assume that $\tau \neq -1$. Then $\varphi$ is normalizable.

\begin{itemize}
 \item The map $$\tilde{\varphi} : A \overset{\id_A \otimes\tilde{\delta}}{\to} A\otimes A\otimes A \overset{m \otimes \id_A }{\to} A\otimes A \overset{m}{\to} A \overset{\varphi}{\to} \C$$ is a $H$-colinear map. Using Schur's lemma, we have $\dim(\Hom_H(A,\C))=1$, hence there exists $c\in \C$ such that $\tilde{\varphi} = c \varphi$. Let us compute $\tilde{\varphi}(1_A)=\tau+1$ as in the proof of Theorem \ref{ClassTh}, so $\tilde{\varphi} = (\tau+1) \varphi$, and $\varphi$ is homogeneous.
 \item We have $\varphi(1_A)=1$, so $\varphi$ is normalizable.
\end{itemize}
To summarize, there exists a finite dimensional, semisimple, measured, normalizable algebra $(A,\varphi)$. According to Remark \ref{RemNorm}, we can assume that $(A,\varphi)$ is normalized.
\end{proof}

A consequence of this proposition is the partial classification result:

\begin{theo}
 Let $H$ be a $SO(3)$-deformation of type $\mathbf{I_{\tau}}$ such that $\tau \neq -1$. Then there exist a finite dimensional, semisimple, measured algebra $(A,\varphi)$ with $\dim A \geq 4$, and a Hopf algebra isomorphism $\Aaut(A,\varphi) \simeq H$
\end{theo}

\begin{proof}
Let us denote by $(W_n^H)_{n\in \N}$ the family of simple $H$-comodules, $W_1^H:=W$. According to Proposition \ref{prop2}, there exist a normalized algebra $(A, \varphi)$, with dimension $\geq 4$, and a Hopf algebra morphism $$ f : \Aaut(A,\varphi) \to H$$ such that $f_*(W^A)\simeq W^H$. According to Theorem \ref{SO3comod}, there exist $q\in \C^*$ and a monoidal equivalence $$\Com(\Aaut(A,\varphi))\simeq^{\otimes}\Com(\gO(SO_{q^{1/2}}(3)))$$ Let us denote by $W_n^A$, $V_n^A$ and $U_n^A$ the $\Aaut(A,\varphi)$-comodules from Theorem \ref{RacUnit}. If $q$ is generic, then we have $f_*(W_n^A)\simeq W_n^H, \ \forall n \in \N$, so $f$ induces a semi-ring isomorphism $\R^+(\Aaut(A,\varphi))\simeq \R^+(H)$, and then by a standard semi-ring argument $f : \Aaut(A,\varphi) \to H$ is a Hopf algebra isomorphism. In the first case where $q$ is not generic, we have $f_*(W_n^A)\simeq W_n^H, \ \forall 1\leq  n \leq N_1-1$. So we get: $$f_*(W_{N_1-1}^A\otimes W_1^A)\simeq W_{N_1-1}^H\otimes W_1^H\simeq W^H_{N_1-2}\oplus W^H_{N_1-1}\oplus W^H_{N_1},$$ but on the other hand, using the simple filtration, we have:  $$f_*(W_{N_1-1}^A\otimes W_1^A)\simeq W_{N_1-1} \oplus f_*(V_1) \oplus  W_{N_1-2}\oplus W_{N_1-1}. $$ This contradicts the uniqueness of the decomposition of a semisimple comodule into a direct sum of simple comodules. In the last case, we have $f_*(U_{2n}^A)\simeq W_n^H, \ \forall 1\leq  n \leq N_1-1$. Then we get:
$$f_*(U^A_{2(N_1-1)}\otimes U^A_2)\simeq W_{N_1-1}^H \otimes W_1^H \simeq W^H_{N_1-2}\oplus W^H_{N_1-1}\oplus W^H_{N_1},$$ but on the other hand we have:
$$f_*(U^A_{2(N_1-1)}\otimes U^A_2)\simeq W_{N_1-2}^H \oplus f_*(U_1^A\otimes V_1^A)\oplus W_{N_1-2}^H.$$ This also contradicts the uniqueness of the decomposition of a semisimple comodule into a direct sum of simple comodules. Then $\Aaut(A,\varphi)$ is cosemisimple, $q$ is generic and $f$ is an isomorphism.
\end{proof}

\section*{Appendix: Proof of lemma \ref{LemTech}}

We begin by a particular case.

\begin{lemmaa}\label{LemmTechBis}
 Let $E,F\in ob(\A)$. Assume that $E_{\lambda}$ is a diagonal matrix for all $\lambda=1,\dots,n_E$, that $F_{\mu}$ is a lower-triangular matrix for all $\mu=1,\dots,n_F$, that $\Tr(E^{-1})=\Tr(F^{-1})$ and $\tr(E_{\lambda})=\tr(F_{\mu})$ for all $\lambda,\mu$. Then the algebra $\A(E,F)$ is nonzero.
\end{lemmaa}

\begin{proof}
 We want to apply the diamond Lemma \cite{Ber}, for which we freely use the definitions and notations of \cite{KS} (although there are a few misprints there). We have to order the monomials $X_{kl,\mu}^{ij,\lambda}$. We order the set of generators with the following order  ($1\leq \lambda, \nu \leq n_E$, $1\leq i,j \leq d^E_{\lambda}$, $1\leq r,s \leq d^E_{\nu}$, $1\leq \mu, \eta \leq n_F$, $1\leq k,l\leq d^F_{\mu}$, $1\leq p,q\leq d^F_{\eta}$)
$$X_{kl,\mu}^{ij,\lambda} < X_{pq,\eta}^{rs,\nu}   \text{   if   } \left\{ \begin{array}{l}
(\lambda, \mu)<(\nu, \eta) \\ (\lambda, \mu)=(\nu, \eta) \text{ and } (i,k)<(r,p)  \\ (\lambda, \mu)=(\nu, \eta),  (i,k)=(r,p) \text{ and } (j,l)>(s,q) \end{array}\right.$$
Then order the set of monomials according to their length, and two monomials of the same length are ordered lexicographically.

Now we can write a nice presentation for $\A(E,F)$ ($d_F:=d_{n_F}^F,d_E:=d_{n_E}^E)$: 

$$
 (DL)\left\{\begin{aligned}
&X^{r1,\nu}_{ij,\lambda}X^{1s,\nu}_{kl,\mu}=\delta_{\lambda \mu} \delta_{jk} X_{il,\mu}^{rs,\nu}-\sum_{t=2}^{d^E_{\nu}} X_{ij,\lambda}^{rt,\nu}X_{kl,\mu}^{ts,\nu} &(1)\\
&X^{ij,\lambda}_{d_Fd_F,n_F}=\delta_{ij}-\sum_{\mu=1}^{n_F-1}\sum_{k=1}^{d_{\mu}^F} X_{kk,\mu}^{ij,\lambda}- \sum_{k<d_F} X_{kk,n_F}^{ij,\lambda}&(2)\\
&X_{ij,\lambda}^{d_Ed_E,n_E}=E_{d_Ed_E,n_E}\big(F^{-1}_{ij,\lambda}-\sum_{\mu=1}^{n_E-1}\sum_{t=1}^{d^E_{\mu}}E^{-1}_{tt,\mu} X_{ij,\lambda}^{tt,\mu}-\sum_{t<d_E}E^{-1}_{tt,n_E} X_{ij,\lambda}^{tt,n_E}\big) &(3) \\
&X_{i1,\lambda}^{kp,\mu}X_{1j,\lambda}^{ql,\nu}=F_{11,\lambda}^{-1}\big(\delta_{\mu \nu}E_{pq,\mu}X_{ij,\lambda}^{kl,\mu}-\sum_{(n,m)\neq(1,1)} F_{nm,\lambda} X_{in,\lambda}^{kp,\mu}X_{mj,\lambda}^{ql,\nu}\big) &(4)\\
\end{aligned} \right. 
$$

Then we have the following inclusion ambiguities: 

$$\begin{array}{lll}
 (X_{i1,\lambda}^{r1,\nu}X_{1j,\lambda}^{1s,\nu}; X_{i1,\lambda}^{r1,\nu}X_{1j,\lambda}^{1s,\nu})  &(X_{ij,\lambda}^{r1,\nu}X_{d_Fd_F,n_F}^{1s,\nu}; X_{d_Fd_F,n_F}^{1s,\nu})  &(X_{d_Fd_F,n_F}^{r1,\nu},X_{d_Fd_F,n_F}^{r1,\nu}X_{ij,\lambda}^{1s,\nu})\\
 & & \\
 (X_{i1,\lambda}^{d_Ed_E,n_E};X_{i1,\lambda}^{d_Ed_E,n_E}X_{1j,\lambda}^{ql,\nu})  &(X_{i1,\lambda}^{kp,\mu}X_{1j,\lambda}^{d_Ed_E,n_E}; X_{1j,\lambda}^{d_Ed_E,n_E})  &(X_{d_Fd_F,n_F}^{d_Ed_E,n_E};X_{d_Fd_F,n_F}^{d_Ed_E,n_E})\\
 & & \\
 (X_{ij,\lambda}^{r1,\nu}X_{kl,\mu}^{11,\nu}; X_{kl,\mu}^{11,\nu}X_{pq,\tau}^{1s,\nu}) &(X^{kp,\mu}_{i1,\lambda}X^{ql,\nu}_{11,\lambda}; X^{ql,\nu}_{11,\lambda}X^{uv,\tau}_{1r,\lambda}) &
\end{array}$$

and the following overlap ambiguities:

\begin{align*} 
 &(X^{r1,\nu}_{ij,\lambda}X^{1s,\nu}_{k1,\mu}; X^{1s,\nu}_{k1,\mu}X^{pq,\eta}_{1l,\mu}) & &(X^{kl,\mu}_{i1,\lambda}X^{r1,\nu}_{1j,\lambda}; X^{r1,\nu}_{1j,\lambda}X^{1s,\nu}_{pq,\eta})
\end{align*}

Let us show that all this ambiguities are resolvable (recall that ''$\to$'' means we perform a reduction):

Let us begin by the ambiguity $(X_{ij,\lambda}^{r1,\nu}X_{kl,\mu}^{11,\nu}; X_{kl,\mu}^{11,\nu}X_{pq,\tau}^{1s,\nu})$. On the first hand, we have:

\begin{flushleft}
$\begin{aligned}
  &\delta_{\lambda \mu}\delta_{jk} X_{il,\mu}^{r1,\nu} X_{pq,\tau}^{1s,\nu} - \sum_{t=2}^{d_\nu^E} X_{ij,\lambda}^{rt,\nu}X_{kl,\mu}^{t1,\nu}X_{pq,\tau}^{1s,\nu}\\
  \overset{(1)}{\to} &\delta_{\lambda \mu} \delta_{\lambda \tau} \delta_{jk}\delta_{lp} X_{iq,\mu}^{rs, \nu} - \delta_{\lambda \mu} \delta_{jk} \sum_{u=2}^{d_\nu^E} X_{il,\mu}^{ru,\nu}X_{pq, \tau}^{us, \nu} - \delta_{\mu \tau}\delta_{lp}\sum_{t=2}^{d_\nu^E} X_{ij,\lambda}^{rt,\nu}X_{kq, \mu}^{ts, \nu} + \sum_{t,u} X_{ij,\lambda}^{rt,\nu} X_{kl,\mu}^{tu,\nu} X_{pq, \tau}^{us, \nu}
 \end{aligned}$
\end{flushleft}
and on the other hand, we have:
\begin{flushleft}
$\begin{aligned}
  &\delta_{\mu \tau} \delta_{lp} X_{ij,\lambda}^{r1,\nu} X_{kq,\mu}^{1s,\nu} - \sum_{u=2}^{d_\nu^E} X_{ij,\lambda}^{r1,\nu} X_{kl,\mu}^{1u,\nu} X_{pq,\tau}^{us, \nu}\\
  \overset{(1)}{\to} &\delta_{\lambda \mu} \delta_{\lambda \tau} \delta_{jk}\delta_{lp} X_{iq,\mu}^{rs, \nu} - \delta_{\lambda \mu} \delta_{jk} \sum_{u=2}^{d_\nu^E} X_{il,\mu}^{ru,\nu}X_{pq, \tau}^{us, \nu} - \delta_{\mu \tau}\delta_{lp}\sum_{t=2}^{d_\nu^E} X_{ij,\lambda}^{rt,\nu}X_{kq, \mu}^{ts, \nu} + \sum_{t,u} X_{ij,\lambda}^{rt,\nu} X_{kl,\mu}^{tu,\nu} X_{pq, \tau}^{us, \nu}
 \end{aligned}$
\end{flushleft}
The ambiguity $(X^{kp,\mu}_{i1,\lambda}X^{ql,\nu}_{11,\lambda}; X^{ql,\nu}_{11,\lambda}X^{uv,\tau}_{1r,\lambda})$ is resolvable by the same kind of computations.

Let us show that the ambiguity $(X_{i1,\lambda}^{r1,\nu}X_{1j,\lambda}^{1s,\nu}; X_{i1,\lambda}^{r1,\nu}X_{1j,\lambda}^{1s,\nu})$ is resolvable. On the first hand, we have:
\begin{flushleft}
$\begin{aligned}
 &F^{-1}_{11,\lambda}\big(-\sum_{(n,m)\neq(1,1)} F_{nm,\lambda} X_{in,\lambda}^{r1,\nu}X_{mj,\lambda }^{1s,\nu} +E_{11,\nu} X_{ij,\lambda}^{rs,\nu}\big)\\
\overset{(2)}{\to} &F_{11,\lambda}^{-1}\big(-\sum_{(n,m)\neq(1,1)} F_{nm,\lambda} \big(-\sum_{t=2}^{d_{\nu}^E} X_{in,\lambda}^{rt,\nu}X_{mj,\lambda}^{ts,\nu} + \delta_{nm}X_{ij,\lambda}^{rs,\nu}\big) + E_{11,\nu} X_{ij,\lambda}^{rs,\nu}\big)\\
= &F_{11,\lambda}^{-1}\big(\sum_{(n,m)\neq(1,1)}\sum_{t=2}^{d_{\nu}^E} F_{nm,\lambda} X_{in,\lambda}^{rt,\nu}X_{mj,\lambda}^{ts,\nu} - (\sum_{(n,m)\neq(1,1)} \delta_{nm} F_{nm,\lambda})X_{ij,\lambda}^{rs,\nu}\big) + E_{11,\nu} X_{ij,\lambda}^{rs,\nu}\big)\\
= &F_{11,\lambda}^{-1}\big(\sum_{(n,m)\neq(1,1)}\sum_{t=2}^{d_{\nu}^E} F_{nm,\lambda} X_{in,\lambda}^{rt,\nu}X_{mj,\lambda}^{ts,\nu} + (F_{11,\lambda}+E_{11,\nu}-\tr (F_{\lambda})) X_{ij,\lambda}^{rs,\nu}\big)
\end{aligned}$
\end{flushleft}
and on the second hand, we have:
\begin{flushleft}
$\begin{aligned}
 &-\sum_{t=2}^{d_{\nu}^E} X_{i1,\lambda}^{rt,\nu}X_{1j,\lambda}^{ts,\nu} + X_{ij,\lambda}^{rs,\nu}\\
\overset{(4)}{\to} &F_{11,\lambda}^{-1} \big(\sum_{t=2}^{d_{\nu}^E}(\sum_{(n,m)\neq(1,1)} F_{nm,\lambda}X_{in,\lambda}^{rt,\nu} X_{mj,\lambda}^{ts,\nu} + E_{tt,\nu}X_{ij,\lambda}^{rs,\nu}) + F_{11,\lambda}X_{ij,\lambda}^{rs,\nu}\big)\\
= &F_{11,\lambda}^{-1}\big(\sum_{t=2}^{d_{\nu}^E}\sum_{(n,m)\neq(1,1)} F_{nm,\lambda} X_{in,\lambda}^{rt,\nu}X_{mj,\lambda}^{ts,\nu} + (F_{11,\lambda}+E_{11,\nu}-\tr (E_{\nu})) X_{ij,\lambda}^{rs,\nu}\big)\\
= &F_{11,\lambda}^{-1}\big(\sum_{t=2}^{d_{\nu}^E}\sum_{(n,m)\neq(1,1)} F_{nm,\lambda} X_{in,\lambda}^{rt,\nu}X_{mj,\lambda}^{ts,\nu} + (F_{11,\lambda}+E_{11,\nu}-\tr (F_{\lambda})) X_{ij,\lambda}^{rs,\nu}\big)
\end{aligned}$
\end{flushleft}
because $\tr(E_{\nu})=\tr(F_{\lambda})$ by assumption.

\amb{$(X_{ij,\lambda}^{r1,\nu}X_{d_Fd_F,n_F}^{1s,\nu}; X_{d_Fd_F,n_F}^{1s,\nu})$}
\begin{flushleft}
$\begin{aligned}
 & -\sum_{\mu=1}^{n_F-1} \sum_{k=1}^{d_{\mu}^F} X_{ij,\lambda}^{r1,\nu}X_{kk,\mu}^{1s,\nu}-\sum_{k<d_F}X_{ij,\lambda}^{r1,\nu}X_{kk,n_F}^{1s,\nu} + \delta_{1s} X_{ij,\lambda}^{r1,\nu}\\
\overset{(1)}{\to} & \sum_{\mu=1}^{n_F-1} \sum_{k=1}^{d_{\mu}^F}(\sum_{t=2}^{d_{\nu}^E} X_{ij,\lambda}^{rt,\nu}X_{kk,\mu}^{ts,\nu} - \delta_{\lambda \mu} \delta_{jk} X_{ik,\lambda}^{rs,\nu}) + \sum_{k<d_F} (\sum_{t=2}^{d_{\nu}^E} X_{ij,\lambda}^{rt,\nu}X_{kk,n_F}^{ts,\nu} - \delta_{\lambda n_F} \delta_{jk} X_{ik,n_F}^{rs,\nu})+ \delta_{1s} X_{ij,\lambda}^{r1,\nu}\\
= & \sum_{\mu=1}^{n_F-1} \sum_{k=1}^{d_{\mu}^F}\sum_{t=2}^{d_{\nu}^E} X_{ij,\lambda}^{rt,\nu}X_{kk,\mu}^{ts,\nu} + \sum_{k<d_F} \sum_{t=2}^{d_{\nu}^E} X_{ij,\lambda}^{rt,\nu}X_{kk,n_F}^{ts,\nu} - \sum_{\mu=1}^{n_F-1} \sum_{k=1}^{d_{\mu}^F}\delta_{\lambda \mu} \delta_{jk} X_{ik,\lambda}^{rs,\nu}  - \sum_{k<d_F} \delta_{\lambda n_F} \delta_{jk} X_{ik,n_F}^{rs,\nu}\\ 
& + \delta_{1s} X_{ij,\lambda}^{r1,\nu}\\
= & \sum_{\mu=1}^{n_F-1} \sum_{k=1}^{d_{\mu}^F}\sum_{t=2}^{d_{\nu}^E} X_{ij,\lambda}^{rt,\nu}X_{kk,\mu}^{ts,\nu} + \sum_{k<d_F} \sum_{t=2}^{d_{\nu}^E} X_{ij,\lambda}^{rt,\nu}X_{kk,n_F}^{ts,\nu} - \sum_{\mu=1}^{n_F} \sum_{k=1}^{d_{\mu}^F}\delta_{\lambda \mu} \delta_{jk} X_{ik,\lambda}^{rs,\nu} \\ 
& + \delta_{\lambda n_F} \delta_{j d_F} X_{id_F,n_F}^{rs,\nu}+ \delta_{1s} X_{ij,\lambda}^{r1,\nu} \\ 
= & \sum_{\mu=1}^{n_F-1} \sum_{k=1}^{d_{\mu}^F} \sum_{t=2}^{d_{\nu}^E}  X_{ij,\lambda}^{rt,\nu} X_{kk,\mu}^{ts,\nu} + \sum_{k<d_F}\sum_{t=2}^{d_{\nu}^E} X_{ij,\lambda}^{rt,\nu} X_{kk,n_F}^{ts,\nu} - \sum_{t=2}^{d_{\nu}^E} \delta_{ts} X_{ij,\lambda}^{rt,\nu} + \delta_{\lambda n_F} \delta_{j d_F} X_{id_F,n_F}^{rs,\nu}
\end{aligned}$
\end{flushleft}
On the other hand, we have:
\begin{flushleft}
$\begin{aligned}
 & -\sum_{t=2}^{d_{\nu}^E} X_{ij,\lambda}^{rt,\nu}X_{d_Fd_F,n_F}^{ts,\nu} + \delta_{\lambda n_F} \delta_{j d_F} X_{id_F,n_F}^{rs,\nu} \\
\overset{(2)}{\to} & \sum_{t=2}^{d_{\nu}^E} X_{ij,\lambda}^{rt,\nu}(\sum_{\mu=1}^{n_F-1} \sum_{k=1}^{d_{\mu}^F} X_{kk,\mu}^{ts,\nu} + \sum_{k<d_F} X_{kk,n_F}^{ts,\nu} - \delta_{ts}) + \delta_{\lambda n_F} \delta_{j d_F} X_{id_F,n_F}^{rs,\nu} 
 \end{aligned}$
 $\begin{aligned}
= & \sum_{t=2}^{d_{\nu}^E} \sum_{\mu=1}^{n_F-1} \sum_{k=1}^{d_{\mu}^F} X_{ij,\lambda}^{rt,\nu} X_{kk,\mu}^{ts,\nu} + \sum_{t=2}^{d_{\nu}^E}\sum_{k<d_F} X_{ij,\lambda}^{rt,\nu} X_{kk,n_F}^{ts,\nu} - \sum_{t=2}^{d_{\nu}^E} \delta_{ts} X_{ij,\lambda}^{rt,\nu} + \delta_{\lambda n_F} \delta_{j d_F} X_{id_F,n_F}^{rs,\nu}\\
\end{aligned}$
\end{flushleft}
The ambiguity $(X_{d_Fd_F,n_F}^{r1,\nu},X_{d_Fd_F,n_F}^{r1,\nu}X_{ij,\lambda}^{1s,\nu})$ is resolvable by the same kind of computations.

\amb{$(X_{i1,\lambda}^{d_Ed_E,n_E};X_{i1,\lambda}^{d_Ed_E,n_E}X_{1j,\lambda}^{ql,\nu})$}
\begin{flushleft}
$\begin{aligned}
 & F^{-1}_{11,\lambda}\big(-\sum_{(n,m)\neq (1,1)}F_{nm,\lambda} X_{in,\lambda}^{d_Ed_E,n_E}X_{mj,\lambda}^{ql,\nu} + \delta_{n_E \nu} E_{d_Eq,\nu}X_{ij,\lambda}^{d_El,\nu}\big)\\
\overset{(3)}{\to} & F^{-1}_{11,\lambda}E_{d_Ed_E,n_E} \big( \delta_{n_E \nu} \delta _{d_E q} X_{ij,\lambda}^{d_El,\nu} + \sum_{(n,m)\neq (1,1)}F_{nm,\lambda} ( \sum_{\mu=1}^{n_E-1}\sum_{t=1}^{d_{\mu}^E} E_{tt,\mu}^{-1} X_{in,\lambda}^{tt,\mu} \\ 
&+ \sum_{t<d_E} E_{tt,n_E}^{-1} X_{in,\lambda}^{tt,n_E} - F_{in,\lambda}^{-1})X_{mj,\lambda}^{ql,\nu}\big) \\
= & F^{-1}_{11,\lambda}E_{d_Ed_E,n_E} \big( \delta_{n_E \nu} \delta _{d_E q} X_{ij,\lambda}^{d_El,\nu} + \sum_{(n,m)\neq (1,1)}\sum_{\mu=1}^{n_E-1}\sum_{t=1}^{d_{\mu}^E} F_{nm,\lambda} E_{tt,\mu}^{-1} X_{in,\lambda}^{tt,\mu} \\ 
& + \sum_{(n,m)\neq(1,1)}\sum_{t<d_E} F_{nm,\lambda}E_{tt,n_E}^{-1} X_{in,\lambda}^{tt,n_E}-\sum_{(n,m)\neq(1,1)}F_{nm,\lambda}F_{in,\lambda}^{-1}X_{mj,\lambda}^{ql,\nu}\big)
\end{aligned}$
\end{flushleft}
and on the other hand:
\begin{flushleft}
$\begin{aligned}
 & E_{d_Ed_E,n_E}\big(F_{i1,\lambda}^{-1} X_{1j,\lambda}^{ql,\nu}-\sum_{\mu=1}^{n_E-1}\sum_{t=1}^{d_{\mu}^E} E^{-1}_{tt,\mu} X_{i1,\lambda}^{tt,\mu} X_{1j,\lambda}^{ql,\nu} - \sum_{t<d_E} E^{-1}_{tt,n_E} X_{i1,\lambda}^{tt,n_E} X_{1j,\lambda}^{ql,\nu}\big)\\
\overset{(4)}{\to} & E_{d_Ed_E,n_E} F^{-1}_{11,\lambda} \big( F_{11,\lambda}F_{i1,\lambda}^{-1} X_{1j,\lambda}^{ql,\nu} + \sum_{\mu=1}^{n_E-1}\sum_{t=1}^{d_{\mu}^E}E_{tt,\mu}^{-1} (\sum_{(n,m)\neq(1,1)} F_{nm,\lambda} X_{in,\lambda}^{tt,\mu} X_{mj,\lambda}^{ql,\nu} + \delta_{\mu \nu} E_{kq,\nu} X_{ij,\lambda}^{tl,\nu}) \\
& + \sum_{t<d_E}E_{tt,n_E}^{-1} (\sum_{(n,m)\neq(1,1)} F_{nm,\lambda} X_{in,\lambda}^{tt,n_E} X_{mj,\lambda}^{ql,\nu} + \delta_{n_E \nu} E_{tq,n_E} X_{ij,\lambda}^{tl,n_E} \big)\\
= & E_{d_Ed_E,n_E} F^{-1}_{11,\lambda} \big( F_{11,\lambda}F_{i1,\lambda}^{-1} X_{1j,\lambda}^{ql,\nu} - \sum_{\mu=1}^{n_E-1}\sum_{t=1}^{d_{\mu}^E} \delta_{\nu \lambda} E^{-1}_{tt,\mu}E_{tq,\nu} X_{ij,\lambda}^{tl,\nu} - \sum_{t<d_E} \delta_{n_E \nu } E^{-1}_{tt,n_E}E_{tq,\nu} X_{ij,\lambda}^{tl,n_E}\\ 
& + \sum_{\mu=1}^{n_E-1}\sum_{t=1}^{d_{\mu}^E}\sum_{(n,m)\neq (1,1)} E_{tt,\mu}^{-1}  F_{nm,\lambda} X_{in,\lambda}^{tt,\mu} X_{mj,\lambda}^{ql,\nu} + \sum_{t<d_E}\sum_{(n,m)\neq(1,1)}E_{tt,n_E}^{-1}  F_{nm,\lambda} X_{in,\lambda}^{tt,n_E} X_{mj,\lambda}^{ql,\nu}\\
= & F^{-1}_{11,\lambda}E_{d_Ed_E,n_E} \big( \delta_{n_E \nu} \delta _{d_E q} X_{ij,\lambda}^{d_El,\nu}  + \sum_{(n,m)\neq (1,1)}\sum_{\mu=1}^{n_E-1}\sum_{t=1}^{d_{\mu}^E} F_{nm,\lambda} E_{tt,\mu}^{-1} X_{in,\lambda}^{tt,\mu} \\
 & + \sum_{(n,m)\neq(1,1)}\sum_{t<d_E} F_{nm,\lambda}E_{tt,n_E}^{-1} X_{in,n_E}^{tt,\mu} -  \sum_{(n,m)\neq(1,1)}F_{nm,\lambda}F_{in,\lambda}^{-1}X_{mj,\lambda}^{ql,\nu}\big)
\end{aligned}$
\end{flushleft}
The ambiguity $(X_{i1,\lambda}^{kp,\mu}X_{1j,\lambda}^{d_Ed_E,n_E}; X_{1j,\lambda}^{d_Ed_E,n_E})$ is resolvable by the same kind of computations.

\amb{$(X_{d_Fd_F,n_F}^{d_Ed_E,n_E};X_{d_Fd_F,n_F}^{d_Ed_E,n_E})$}
\begin{flushleft}
$\begin{aligned}
 &-\sum_{\mu=1}^{n_F-1} \sum_{k=1}^{d_{\mu}^F} X_{kk,\mu}^{d_Ed_E,n_E} - \sum_{k<d_F} X_{kk,n_F}^{d_Ed_E,n_E} + 1 \\
\overset{(3)}{\to} &E_{d_Ed_E,n_E}\big(-\sum_{\mu=1}^{n_F-1} \sum_{k=1}^{d_{\mu}^F}(-\sum_{\lambda=1}^{n_E-1}\sum_{t=1}^{d_{\lambda}^E} E_{tt,\lambda}^{-1} X_{kk,\mu}^{tt,\lambda} - \sum_{t< d_E} E_{tt,n_E}^{-1}X_{kk,\mu}^{tt,n_E} + F^{-1}_{kk,\mu}) \\
 & -\sum_{k< d_F}(-\sum_{\lambda=1}^{n_E-1}\sum_{t=1}^{d_{\lambda}^E} E_{tt,\lambda}^{-1} X_{kk,n_F}^{tt,\lambda} - \sum_{t< d_E} E_{tt,n_E}^{-1}X_{kk,n_F}^{tt,n_E} + F^{-1}_{kk,n_F}) + E_{d_Ed_E,n_E}^{-1} \big)
 \end{aligned}$
\end{flushleft}
\begin{flushleft}
 $\begin{aligned}
 = &E_{d_Ed_E,n_E}\big(\sum_{\mu=1}^{n_F-1} \sum_{k=1}^{d_{\mu}^F} \sum_{\lambda=1}^{n_E-1}\sum_{t=1}^{d_{\lambda}^E} E_{tt,\lambda}^{-1} X_{kk,\mu}^{tt,\lambda} + \sum_{\mu=1}^{n_F-1} \sum_{k=1}^{d_{\mu}^F} \sum_{t< d_E} E_{tt,n_E}^{-1}X_{kk,\mu}^{tt,n_E} - \sum_{\mu=1}^{n_F-1} \tr(F^{-1}_{\mu}) \\
 & +\sum_{k< d_F}\sum_{\lambda=1}^{n_E-1}\sum_{t=1}^{d_{\lambda}^E} E_{tt,\lambda}^{-1} X_{kk,n_F}^{tt,\lambda} - \sum_{k< d_F} \sum_{t< d_E} E_{tt,n_E}^{-1}X_{kk,n_F}^{tt,n_E} + F^{-1}_{d_Fd_F,n_F} - \tr(F^{-1}_{n_F}) + E_{d_Ed_E,n_E}^{-1}\big)\\
 = &E_{d_Ed_E,n_E}\big(\sum_{\mu=1}^{n_F-1} \sum_{k=1}^{d_{\mu}^F}\sum_{\lambda=1}^{n_E-1}\sum_{t=1}^{d_{\lambda}^E} E_{tt,\lambda}^{-1} X_{kk,\mu}^{tt,\lambda} + \sum_{\mu=1}^{n_F-1} \sum_{k=1}^{d_{\mu}^F} \sum_{t< d_E} E_{tt,n_E}^{-1}X_{kk,\mu}^{tt,n_E} +\sum_{k< d_F}\sum_{\lambda=1}^{n_E-1}\sum_{t=1}^{d_{\lambda}^E} E_{tt,\lambda}^{-1} X_{kk,n_F}^{tt,\lambda}\\
& - \sum_{k< d_F} \sum_{t< d_E} E_{tt,n_E}^{-1}X_{kk,n_F}^{tt,n_E}  + F^{-1}_{d_Fd_F,n_F} + E_{d_Ed_E,n_E}^{-1} - \Tr(F^{-1})\big)
\end{aligned}$
\end{flushleft}
and on the other hand:
\begin{flushleft}
$\begin{aligned}
& E_{d_Ed_E,n_E}\big( -\sum_{\lambda=1}^{n_E-1}\sum_{t=1}^{d_{\mu}^E} E^{-1}_{tt,\lambda} X_{d_Fd_F,n_F}^{tt,\lambda} - \sum_{t<d_E} E^{-1}_{tt,n_E} X_{d_Fd_F,n_F}^{tt,n_E} + F^{-1}_{d_Fd_F,n_F}\big)\\
\overset{(2)}{\to} & E_{d_Ed_E,n_E} \big(-\sum_{\lambda=1}^{n_E-1}\sum_{t=1}^{d_{\mu}^E} E^{-1}_{tt,\lambda}(-\sum_{\mu=1}^{n_F-1} \sum_{k=1}^{d_{\mu}^F} X_{kk,\mu}^{tt,\lambda} -\sum_{k<d_F} X_{kk,n_F}^{tt,\lambda} + 1) \\ 
&-\sum_{t<d_E} E^{-1}_{tt,n_E} (-\sum_{\mu=1}^{n_F-1} \sum_{k=1}^{d_{\mu}^F} X_{kk,\mu}^{tt,n_E} -\sum_{k<d_F} X_{kk,n_F}^{tt,n_E} + 1) + F^{-1}_{d_Fd_F,n_F} \big)\\
= & E_{d_Ed_E,n_E}\big(\sum_{\mu=1}^{n_F-1} \sum_{k=1}^{d_{\mu}^F}\sum_{\lambda=1}^{n_E-1}\sum_{t=1}^{d_{\mu}^E} E_{tt,\lambda}^{-1} X_{kk,\mu}^{tt,\lambda} + \sum_{\mu=1}^{n_F-1} \sum_{k=1}^{d_{\mu}^F} \sum_{t< d_E} E_{tt,n_E}^{-1}X_{kk,\mu}^{tt,n_E}   +\sum_{k< d_F}\sum_{\lambda=1}^{n_E-1}\sum_{t=1}^{d_{\mu}^E} E_{tt,\lambda}^{-1} X_{kk,n_F}^{tt,\lambda} \\
 &- \sum_{k< d_F} \sum_{t< d_E} E_{tt,n_E}^{-1}X_{kk,n_F}^{tt,n_E} + F^{-1}_{d_Fd_F,n_F} + E_{d_Ed_E,n_E}^{-1} - \Tr(E^{-1})\big)\\
= & E_{d_Ed_E,n_E}\big(\sum_{\mu=1}^{n_F-1} \sum_{k=1}^{d_{\mu}^F}\sum_{\lambda=1}^{n_E-1}\sum_{t=1}^{d_{\mu}^E} E_{tt,\lambda}^{-1} X_{kk,\mu}^{tt,\lambda} + \sum_{\mu=1}^{n_F-1} \sum_{k=1}^{d_{\mu}^F} \sum_{t< d_E} E_{tt,n_E}^{-1}X_{kk,\mu}^{tt,n_E}+\sum_{k< d_F}\sum_{\lambda=1}^{n_E-1}\sum_{t=1}^{d_{\mu}^E} E_{tt,\lambda}^{-1} X_{kk,n_F}^{tt,\lambda}\\ & 
 - \sum_{k< d_F} \sum_{t< d_E} E_{tt,n_E}^{-1}X_{kk,n_F}^{tt,n_E} + F^{-1}_{d_Fd_F,n_F} + E_{d_Ed_E,n_E}^{-1} - \Tr(F^{-1})\big)
\end{aligned}$
\end{flushleft}
because $\Tr(E^{-1})=\Tr(F^{-1})$ by assumption.

Now, \Amb{$(X^{r1,\nu}_{ij,\lambda}X^{1s,\nu}_{k1,\mu}; X^{1s,\nu}_{k1,\mu}X^{pq,\eta}_{1l,\mu})$}
\begin{flushleft}
$\begin{aligned}
 & F^{-1}_{11,\mu}\big(-\sum_{(n,m)\neq (1,1)} F_{nm,\mu} X_{ij,\lambda}^{r1,\nu} X_{kn,\mu}^{1s,\nu} X_{ml,\mu}^{pq,\eta} + \delta_{\nu \eta} E_{sp,\nu} X_{ij,\lambda}^{r1,\nu}X_{kl,\mu}^{1q,\nu}\big)\\
\overset{(1)}{\to} & F^{-1}_{11,\mu}\big(\sum_{(n,m)\neq (1,1)} F_{nm,\mu}(\sum_{t=2}^{d_{\nu}^E}X_{ij,\lambda}^{rt,\nu}X_{kl,\mu}^{tq,\nu} - \delta_{\lambda \mu} \delta_{jk} X_{in,\mu}^{rs,\nu})X_{ml,\mu}^{pq,\eta} \\
    & - \delta_{\nu \eta} E_{sp,\nu}(\sum_{t=2}^{d_{\nu}^E}X_{ij,\lambda}^{rt,\nu}X_{kl,\mu}^{tq,\nu} - \delta_{\lambda \mu} \delta_{jk} X_{il,\mu}^{rq,\nu})\big)
 \end{aligned}$
\end{flushleft}
\begin{flushleft}
 $\begin{aligned}
 = &F^{-1}_{11,\mu}\big(\delta_{\nu \eta} \delta_{\lambda \mu} \delta_{jk} E_{sp,\nu}X_{il,\mu}^{rq,\nu} - \delta_{\nu \eta} E_{sp,\nu}\sum_{t=2}^{d_{\nu}^E}X_{ij,\lambda}^{rt,\nu}X_{kl,\mu}^{tq,\nu} \\
 & - \delta_{\lambda \mu} \delta_{jk} \sum_{(n,m)\neq (1,1)} F_{nm,\mu}X_{in,\mu}^{rs,\nu}X_{ml,\mu}^{pq,\eta}+\sum_{(n,m)\neq (1,1)} \sum_{t=2}^{d_{\nu}^E}F_{nm,\mu}X_{ij,\lambda}^{rt,\nu}X_{kn,\mu}^{ts,\nu}X_{ml,\mu}^{pq,\eta} 
\end{aligned}$
\end{flushleft}
And on the other hand, we have:
\begin{flushleft}
$\begin{aligned}
 & \delta_{\lambda \mu} \delta_{jk} X_{i1,\mu}^{rs,\nu} X^{pq,\eta}_{1l,\mu} - \sum_{t=2}^{d_{\nu}^E} X_{ij,\lambda}^{rt,\nu}X_{k1,\mu}^{ts,\nu}X^{pq,\eta}_{1l,\mu}\\
\overset{(4)}{\to} & \delta_{\lambda \mu} \delta_{jk} F_{11,\mu}^{-1} (\sum_{(n,m)\neq(1,1)} F_{nm,\mu} X_{in,\lambda}^{rs,\nu} X_{ml,\mu}^{pq,\eta} + \delta_{\nu \eta} E_{sp,\nu} X_{il,\mu}^{rq,\nu}) \\
& + F_{11,\mu}^{-1}(\sum_{t=2}^{d_{\nu}^E} \sum_{(n,m)\neq(1,1)} F_{nm,\mu} X_{ij,\lambda}^{rt,\nu}X_{kn,\mu}^{ts,\nu} X_{ml,\mu}^{pq,\eta} + \delta_{\nu \eta} E_{sp,\nu} X_{ij,\lambda}^{rt,\nu}X_{kl,\mu}^{tq,\nu})\\
 = &F^{-1}_{11,\mu}\big(\delta_{\nu \eta} \delta_{\lambda \mu} \delta_{jk} E_{sp,\nu}X_{il,\mu}^{rq,\nu} - \delta_{\nu \eta} E_{sp,\nu}\sum_{t=2}^{d_{\nu}^E}X_{ij,\lambda}^{rt,\nu}X_{kl,\mu}^{tq,\nu} \\
 & - \delta_{\lambda \mu} \delta_{jk} \sum_{(n,m)\neq (1,1)} F_{nm,\mu}X_{in,\mu}^{rs,\nu}X_{ml,\mu}^{pq,\eta} +\sum_{(n,m)\neq (1,1)} \sum_{t=2}^{d_{\nu}^E}F_{nm,\mu}X_{ij,\lambda}^{rt,\nu}X_{kn,\mu}^{ts,\nu}X_{ml,\mu}^{pq,\eta} 
\end{aligned}$
\end{flushleft}
The last ambiguity $(X^{kl,\mu}_{i1,\lambda}X^{r1,\nu}_{1j,\lambda}; X^{r1,\nu}_{1j,\lambda}X^{1s,\nu}_{pq,\eta})$ is resolvable by the same kind of computations.

Then all the ambiguities are resolvable. According to the diamond lemma, the set of reduced monomials forms a linear basis of $\A(E,F)$. In particular, the algebra $\A(E,F)$ is nonzero.
\end{proof}

We have the following isomorphism:

\begin{lemmaa}
 Let $E,P\in \underset{\lambda=1}{\overset{n_E}{\bigoplus}} GL_{d_{\lambda}^E}(\C)$, $F,Q\in \underset{\lambda=1}{\overset{n_F}{\bigoplus}} GL_{d_{\lambda}^F}(\C)$. Then the algebras
$\A(E,F)$ and $\A(PEP^{-1},QFQ^{-1})$ are isomorphic.
\end{lemmaa}

\begin{proof}
Let us denote by $Y_{ij,\lambda}^{kl,\mu}$ the generators of $\A(PEP^{-1},QFQ^{-1})$. They satisfy the relations:
\begin{align*}
 &\sum_{q=1}^{d^E_{\nu}} Y_{ij,\lambda}^{rq,\nu}Y_{kl,\mu}^{qs,\nu}=\delta_{\lambda \mu} \delta_{jk} Y_{il,\mu}^{rs,\nu}, & & \sum_{\mu=1}^{n_F}\sum_{k=1}^{d^F_{\mu}} Y_{kk,\mu}^{ij,\lambda}=\delta_{ij},\\
 &\sum_{\mu=1}^{n_E}\sum_{k=1}^{d^E_{\mu}}(PEP^{-1})^{-1}_{kl,\mu} Y_{ij,\lambda}^{kl,\mu}=(QFQ^{-1})^{-1}_{ij,\lambda}, & & \sum_{r,s=1}^{d^F_{\lambda}} (QFQ^{-1})_{rs,\lambda} Y_{ir,\lambda}^{kp,\mu}Y_{sj,\lambda}^{ql,\nu}=\delta_{\mu \nu}(PEP^{-1})_{pq,\mu}Y_{ij,\lambda}^{kl,\mu}.
\end{align*}
This ensures the existence of an algebra morphism $f: \A(E,F) \to \A(PEP^{-1},QFQ^{-1})$ by setting $$f(X_{ij,\lambda}^{kl,\mu})=\sum_{r,s=1}^{d_{\mu}^E}\sum_{u,v=1}^{d_{\lambda}^F} P_{uk,\mu}P^{-1}_{lv,\mu} Y_{rs,\lambda}^{uv,\mu} Q^{-1}_{ir,\lambda}Q_{sj,\lambda}.$$
The inverse map is given by $$f^{-1}(Y_{ij,\lambda}^{kl,\mu})=\sum_{r,s=1}^{d_{\lambda}^F}\sum_{u,v=1}^{d_{\mu}^E} P^{-1}_{uk,\mu}P_{lv,\mu} X_{rs,\lambda}^{uv,\mu} Q_{ir,\lambda}Q^{-1}_{sj,\lambda}.$$
\end{proof}

We can now prove Lemma \ref{LemTech}.

\begin{proof}
 Let $E\in \underset{\lambda=1}{\overset{n_E}{\bigoplus}} GL_{d_{\lambda}^E}(\C)$, $F\in \underset{\mu=1}{\overset{n_F}{\bigoplus}} GL_{d_{\mu}^F}(\C)\in ob(\A)$ be such that $\Tr(E^{-1})=\Tr(F^{-1})$ and $\tr(E_{\lambda})=\tr(F_{\mu})$ for all $\lambda,\mu$. According to the previous lemma, let $P\in \underset{\lambda=1}{\overset{n_E}{\bigoplus}} GL_{d_{\lambda}^E}(\C)$ and $Q\in \underset{\mu=1}{\overset{n_F}{\bigoplus}} GL_{d_{\mu}^F}(\C)$ be such that $PEP^{-1}$ and $QFQ^{-1}$ are lower-triangular and let $M\in \underset{\lambda=1}{\overset{n_E}{\bigoplus}} GL_{d_{\lambda}^E}(\C)$ be diagonal such that $\Tr(E^{-1})=\Tr(M^{-1})$ and $\tr(E_{\lambda})=\tr(M_{\mu})$ for all $\lambda,\mu$. According to Lemma~\ref{LemmTechBis}, the algebra $\A(M,QFQ^{-1})$ is nonzero, and so is $\A(M,F)$. According to \cite{Bi2}, Proposition 2.15, $\A(E,F)$ is nonzero.
\end{proof}

\begin{rem}
We could have defined a bigger cogroupoid $\widehat{\A}$ such that $ob(\widehat{\A})=\{E\in \underset{\lambda=1}{\overset{n_E}{\bigoplus}} GL_{d_{\lambda}^E}(\C) \}$. In that case, $F=(F_1,\dots, F_n) \in (\C^*)^n$ is normalized if and only if $F_1=\dots =F_n=f$, and $\Aaut(A_F,\tr_F)=\Aaut(C(X_n), \psi)$ where $X_n=\{x_1, \dots, x_n\}$. The previous proof no longer works because the relations lead to more ambiguities, which are no longer resolvable.
\end{rem}

\subsection*{Acknowledgments}

The author is very grateful to J. Bichon for his advice and his careful proofreading. We also thank T. Banica for pointing the paper \cite{GrossSny}.

\bibliographystyle{plain}
\bibliography{bibliographie.bib}

%

\end{document}